\documentclass[12pt]{article}

\usepackage[utf8]{inputenc} % allow utf-8 input
\usepackage[T1]{fontenc}    % use 8-bit T1 fonts
\usepackage{hyperref}       % hyperlinks
\usepackage{url}            % simple URL typesetting
\usepackage{booktabs}       % professional-quality tables
\usepackage{amsfonts}      
\usepackage{nicefrac}       % compact symbols for 1/2, etc.
\usepackage{microtype}      % microtypography
\usepackage{lipsum}
\usepackage{graphicx}
\usepackage{natbib}
\graphicspath{ {./images/} }
\usepackage{smile_old}
\usepackage{geometry}
\def\T{{ \mathrm{\scriptscriptstyle T} }}

\title{Inferring serial correlation with dynamic backgrounds}

\author{Song Wei$^\mathrm{a}$ \quad Yao Xie$^\mathrm{a}$ \quad Dobromir Rahnev$^\mathrm{b}$\\
  $^\mathrm{a}$School of Industrial and Systems Engineering,\\ $^\mathrm{b}$School of Psychology,\\
Georgia Institute of Technology, Atlanta, Georgia,\\ 30332-0205, U.S.A. }

\begin{document}
\maketitle
\begin{abstract}
Sequential data with serial correlation and an unknown, unstructured, and dynamic background is ubiquitous in neuroscience, psychology, and econometrics. Inferring serial correlation for such data is a fundamental challenge in statistics. We propose a total variation constrained least square estimator coupled with hypothesis tests to infer the serial correlation in the presence of unknown and unstructured dynamic background. The total variation constraint on the dynamic background encourages a piece-wise constant structure, which can approximate a wide range of dynamic backgrounds. The tuning parameter is selected via the Ljung-Box test to control the bias-variance trade-off. We establish a non-asymptotic upper bound for the estimation error through variational inequalities. We also derive a lower error bound via Fano's method and show the proposed method is near-optimal. Numerical simulation and a real study in psychology demonstrate the excellent performance of our proposed method compared with the state-of-the-art.
\end{abstract}


\section{Introduction}\label{sec:intro}

Serial correlation and serial dependence have been central to time series analysis \citep{hong2010serial}. Modern time-series data from neuroscience, psychology, and economics usually contain both a substantial serial dependence and a non-stationary drift  \citep{Akrami2018, Wexler2015, Moskowitz2012, Fischer2014, Cicchini2018, McIlhagga2008, Rahnev2015}. A well-known example comes from human reaction times, which are thought to be autocorrelated but also drift throughout an experiment \citep{Laming1968}. The drift can be due to many factors such as becoming better on the task, increased tiredness, and attention or arousal fluctuations. None of these influences take a specific parametric form. While some (e.g., learning or fatigue) are likely to be monotonic, others (e.g., fluctuations in attention) can be expected to waver unpredictably. This non-stationary background drift is thus typically considered a nuisance variable.
% when studying the autocorrelation in the time series.

It is typically of strong scientific interest to infer the presence and/or assess serial correlation's strength with an unknown and unstructured dynamic background. The magnitude of autocorrelation has direct implications for many scientific theories. For example, \citet{Fischer2014} proposed that the human brain creates a ``perceptual continuity field'' where the subjective percept at one point of time directly influences the percept within a subsequent 15-second window. Such effects are known as “serial dependence” and are an active area of research within psychology and neuroscience. Progress in this and related endeavors depends on one’s ability to estimate the magnitude of autocorrelation in certain time series, even in the presence of substantial unstructured drift. 

The most popular tool to handle the serial correlation is the autoregressive time series. However, the presence of even a small drift can induce strong biases in the autoregressive coefficients. For example, unmodeled background drift can masquerade as autocorrelation, as illustrated in the first panel in Figure~\ref{fig:method_illus}. This issue has been pointed out before by \citet{Dutilh2012}, but no solution  exists to date. Techniques have been developed for tracking the unknown dynamic background with minimum structural assumptions \citep{hodrick1997postwar,kim2009ell_1,harchaoui2010multiple} but these approaches do not estimate the serial correlation. {\it Thus, we currently lack an efficient method to capture the autocorrelation strength in a time series in the presence of highly unstructured dynamic drifts.}

\begin{figure}[!htp]
\centering
\subfigure{\includegraphics[width=1\linewidth]{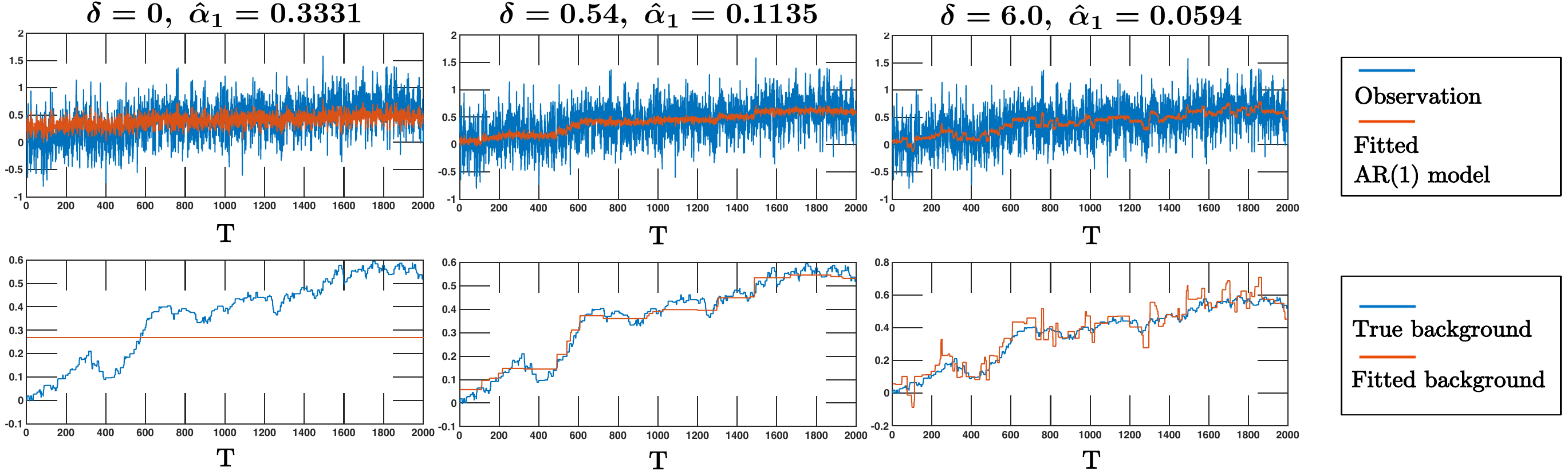}}
    %\vspace{-.1in}
    \caption{An example showing that proper modeling of dynamic background is important in capturing serial correlation. The estimate $\hat \alpha_1$ is specified on the top of each column, with the ground truth $\alpha_1 = 0.1$; $\delta$ is a hyperparameter that controls the data fit and model complexity. In the first panel, we directly fit an \textsc{ar}$(1)$ model while ignoring the dynamic background,  leading to over-estimating $\hat \alpha_1$. In the third panel, the result overfits the dynamic background, leading to underestimating the autoregressive coefficient. The second panel is the desired result obtained by our method. % {\color{blue}Should we mention what $\delta$ is in the figure?} %\color{red} Song: Maybe not, since it will confuse people. Or we can add: Here, $\delta$ is a hyperparameter which balances least square loss and model complexity.
    }\label{fig:method_illus}
    %\vspace{-0.1in}
\end{figure}

% Other experimental settings are $\sigma_0^2 = 0.1$, $\delta_0 = 0.05$, $s = 300$ and $T=2000$. The dynamic background generating mechanism is detailed in Section~\ref{exp_dynamic_background_generate}. Here it just serves as an illustrative example.

%\textbf{A fundamental formulation.}
Motivated by this, we consider the following problem. Assume a sequence of observations $x_1, \ldots, x_T$ over time horizon $T$, which are generated from the underlying non-stationary \textsc{ar}$(p)$ time series model: 
\begin{equation}\label{DGP}
%\begin{array}{rl}
    x_i = f_i + \sum_{j=1}^p \alpha_j x_{i-j}  + \varepsilon_i, \quad i = 1, \ldots, T,
    %\end{array}
\end{equation}
where $\varepsilon_1,\dots,\varepsilon_T$ are i.i.d. sub-Gaussian random noise with zero mean and variance $\sigma_0^2$, $ \alpha_1,\dots,\alpha_p$ are  autoregressive coefficients, $f_1,\dots,f_T$ are deterministic dynamic background and $x_{-p+1},\dots,x_0$ are the known history. The goal is to infer the presence and/or estimate the unknown autoregressive coefficients and dynamic background simultaneously from data. {\it To ensure our model is general, we do not impose parametric or distributional assumptions on the dynamic background $f_i$'s.}

In this paper, we present a new convex optimization based method to estimate the autoregressive coefficients for sequential data in the presence of unknown dynamic background, coupled with the Ljung-Box test for model diagnosis. We cast the problem as minimizing the least square error with a total-variation constraint on the dynamic background, which encourages a piecewise constant structure and can approximate a wide range of unstructured drifts with good precision. We establish performance guarantees for the $\ell_2$ recovery error of the coefficients. To efficiently tune hyperparameters to control the bias-variance trade-off, we adopt the Ljung-Box test \citep{ljung1978measure}. Extensive numerical experiments are performed to validate the effectiveness of the proposed method. We also test our method on a real psychology dataset to demonstrate it can infer whether or not there is a statistically significant correlation. 

The rest of the paper is organized as follows. In the remainder of this section, we discuss related works. Section~\ref{sec:estimation} presents the proposed method. Section~\ref{sec:theory} contains the main theoretical results, including a non-asymptotic bounds on the $\ell_2$ estimation error for the \textsc{ar}$(1)$ model, for the ease of presentation. We discuss how to extend the result to \textsc{ar}$(p)$ models in Section~\ref{sec:extend}.  Section~\ref{sec:numerical} contains simulation results to demonstrate the good performance of our method and validates theoretical results. Section~\ref{sec:real_data} presents a real-data study from a psychology experiment. Finally, Section~\ref{sec:discussion} summaries the paper.

% we apply on the residual sequence and use the $p$-value to quantify the remaining serial correlation. Then the hyperparameter in the aforementioned convex program is selected to be the one with least serial correlation in the residual sequence (maximum $p$-value).
% % We use minimum amount of serial correlation in the residual sequence as the hyperparameter tuning principle in this convex program, which is achieved by Ljung-Box test. 
% This hyperparameter tuning procedure via Ljung-Box test guarantees that we avoid underfitting (left panel in Figure~\ref{fig:method_illus}) and overfitting (right panel in Figure~\ref{fig:method_illus}) and yields an accurate estimate of the autoregressive coefficients (second panel in Figure~\ref{fig:method_illus}). 
% Confidence interval is constructed via bootstrapping with adaption to non-stationarity.

\subsection{Related work}\label{subsec:literature}

%The problem we are interesed in \eqref{DGP} is very fundamental and there have been many existing works considering (special cases of) this formulation.The most well-known model would be regular autoregressive time series, where the dynamic background is not modelled. Statistical inference with uncertain quantification (e.g., confidence intervals of autoregressive coefficients) is critical to many scientific data studies to draw statistically significant claims. However, the problem of non-stationarity can become an interference to the inference problem. 

Standard time series models \citep{brockwell1991time} such as autoregressive and moving average models do not include dynamic backgrounds. On the other hand, the classical approach to capture dynamic background usually makes strong structural assumptions such as the linear trend, periodical trend  \citet{clark1987cyclical} or hidden Markov model \cite{hamilton1989new}. Our problem involves a highly unstructured background. This requires new solution approach; moreover, existing theory does not apply because the unstructured dynamic background leads to a non-stationary time series, which does not satisfy the strong-mixing condition. This disables us from using asymptotic results in the classic time series literature. 
%Since misspecification in modeling the trend will seriously affect the modeling results \citep{nelson1981spurious,nelson1984pitfalls}, these works do not apply to our study.

%(e.g., \citet{gb08})
% {\bf Recent developments.}
Recent works for similar problems also use convex optimization to fit the dynamic background while making few structural assumptions. This line of work typically considers solving a least square problem with various penalties or constraints to encourage desired structures on the fitted background, which can approximate the unknown ground-truth. For instance, H-P filter \citep{hodrick1997postwar} imposed $\ell_2$ penalty on the second-order difference to encourage a smooth background;  \citet{kim2009ell_1} considered a variant of H-P filter with an $\ell_1$ regularization function to capture a piece-wise linear background. Another related work \citep{harchaoui2010multiple} considered change-point detection in the means using least square estimation with total variation penalty; since the number of change points is unknown, the work essentially estimates a piece-wise constant background. While many advances have been achieved, these existing works have not considered serial correlation together with the dynamic background.

Our proposed method is  related to variable fusion  \citep{land1997variable} and fused lasso \citep{tibshirani2005sparsity}. Here the unstructured, dynamic background leads to a high-dimensional problem: we have $T$ equations and $T+p$ variables; the optimal solution is not unique. Thus, we borrow the analytical technique in analyzing high-dimensional lasso, particularly the restricted eigenvalue conditions for the design matrix \citep{bickel2009simultaneous,meinshausen2009lasso,van2009conditions} to derive the theoretical results, while further exploiting the special structure of our design matrix.

There are two closely related recent works: \citet{xu2008bootstrapping} used polynomials to approximate the dynamic background, and \citet{zhang2020real} developed an online forecasting algorithm based on least square estimation with $\ell_2$ variable fusion constraint. These works do not explicitly consider highly unstructured backgrounds. We compare with both methods via numerical simulations in Section~\ref{sec:numerical} and show the advantage of our approach when there are dynamic, unstructured backgrounds; moreover, we also present a method for hyperparameter selection based on the Ljung-Box test.

\section{Proposed Method}\label{sec:estimation}

%We now present a method to separate the dynamic background from the autoregressive time-series part, using a total variation constrained least square approach. Moreover, we also present how to set the important hyperparameter $\delta$ using the Ljung-Box test, and since estimating the autoregressive coefficient is our primary interest, we show how to construct confident interval using bootstrap.

\subsection{Total variation constrained least square estimation}

%We proposed to a total variation norm constrained method  to encourage a piece-wise constant structure on dynamic background vector with few changes. 
Consider a {\it total variation constrained least square estimator} to estimate the autoregressive coefficients and the dynamic background simultaneously, which is obtained by solving the following convex optimization problem:
\begin{equation}\label{cvx_opt_1}
\begin{array}{rl}
\underset{\alpha_1\dots,\alpha_p,f_1,\dots,f_T}{\mbox{minimize}} & \frac{1}{2T} \sum_{i=1}^T \left(x_i - \sum_{j=1}^p \alpha_j x_{i-j} - f_i\right)^2\\
\mbox{subject to} & \sum_{i=1}^{T-1} |f_{i+1} - f_{i}| < \delta,
\end{array}
\end{equation}
where $\delta$ is a user-specified hyperparameter (the selection of $\delta$ is discussed in Section \ref{subsec:tuning}).

As discussed for the problem formulation \eqref{DGP}, different from the conventional autoregressive model, here we consider an {\it unknown} and {\it time-varying background}. Since the number of observations and the number of parameters both grow at the same rate as the time horizon $T$ increases, we cannot uniquely recover the parameters using the available observations. Thus, we impose a total variation constraint on the dynamic background, essentially choosing one solution with the smallest variations. Such an approach can serve as a good approximation to a broad class of unstructured, dynamic backgrounds.

% {\it To make the model general enough, we do not impose any distributional or structural assumption on this dynamic background except that the one-step changes of the dynamic background $\Delta_i$'s are sparse and have bounded magnitude.} Or rather, we can also state that only when the rate for sporadic changes $s$ and their magnitudes $\delta_0$ satisfy some certain constraints can we obtain an accurate estimate of the unknown coefficients. we will elaborate on those constraints in Section~\ref{sec:theory}. Here, we first formally define the rate for sparse changes and their magnitudes.

%one-step changes $\delta_0$. As we have mentioned above, we will fit a piecewise constant dynamic background to approximate the potentially unstructured true one. This requires the true dynamic background not to vary too much in magnitude within some small time windows, in which we need to approximate it by a constant. Thus, the maximum (or magnitude) of one-step changes $\Delta_i$'s should be bounded. We define the magnitude of $\Delta_i$'s as $\delta_0 \geq \max_{i \in \{ 2,\dots,T\}} |\Delta_i|.$ Equivalently, we say that $\Delta_i$'s have bounded magnitude $\delta_0$ if\begin{equation}\label{slowly_varying}
%   |\Delta_i| \leq \delta_0, \quad i = 2,\dots,T.
% \end{equation}

% If we reformulate this into penalized form using Lagrangian Multiplier, we can see it is a modified version of lasso. we will elaborate on this in next section.

\subsection{Hyperparameter tuning procedure}\label{subsec:tuning} 

%we will also elaborate on the assumptions for $s$ and $\delta_0$ to guarantee an accurate estimate.

We will show that the choice of the hyper-parameter $\delta$ in (\ref{cvx_opt_1}) will critically impact its solution (the recovered dynamic background and the AR coefficient). As illustrated in the first panel in Figure~\ref{fig:method_illus}, setting $\delta = 0$ will result in a very simple  \textsc{ar}$(1)$ model, but a very biased estimate $\hat \alpha_1$. Clearly, this model under-fits data. On the other hand, the third panel in Figure~\ref{fig:method_illus} shows that when $\delta$ is too large, the fitted model will have a small empirical loss but overfitted background, which still results in very biased $\hat \alpha_1$. From the second panel in Figure~\ref{fig:method_illus}, we can see the fitted piecewise constant background faithfully captures the dynamics, and this model yields a very accurate estimate $\hat \alpha_1$.  

Figure~\ref{fig:method_illus} illustrates that $\delta$ controls the bias-variance trade-off: a larger $\delta$ leads to a smaller fitting error, but an overfitted background, and thus the estimated AR coefficients are biased. Therefore, we cannot use the fitting error to tune $\delta$. Instead, we choose $\delta$ by the Ljung-Box test, which can test the model's goodness-of-fit by checking the remaining serial correlation in the residual sequence. This test provides a $p$-value to quantify the goodness-of-fit. Since larger $p$-value indicates less remaining serial correlation in the residuals, we select $\delta$ with the maximum $p$-value. Details of parameter tuning procedure can be found in Appendix~\ref{appendix:tuning_CI}.

Here we want to comment that we cannot use the popular cross-validation technique to choose $\delta$. The cross-validation splits the data into training data and testing data. Typically, the model for the training data and the test data are identical; thus, cross-validation error can be used to estimate the actual test error. However, here since our background is dynamic and different on the test and the training data, we cannot apply model fitted on training data to the test data to tune hyperparameter.

\subsection{Bootstrap confidence interval}

Finally, we present two bootstrap methods to construct confidence intervals for the autoregressive coefficients. For serially correlated data, we cannot use the conventional bootstrap for i.i.d. data \citep{efron1992bootstrap} but instead using the following techniques: (i) Wild bootstrap \citep{wu1986jackknife}, which resamples from the fitted residuals and (ii) a variant \citep{kunsch1989jackknife} of local moving block bootstrap, which is designed for non-stationary time series. Details of both methods can be found in Appendix~\ref{appendix:tuning_CI}.

%Besides, we should mention that under Gaussian random noise assumption, the proposed least square estimator is equivalent to total variation constrained maximum likelihood estimator (MLE). However, constructing confidence interval for autoregressive coefficients using asymptotic normality of MLE involves evaluating the estimation error of dynamic background. Bounding this error will require knowledge on $s$ and $\delta_0$, which are typically unknown in practice. Thus, this method is not applicable to our setting.

\section{Non-asymptotic bounds for \textsc{ar}$(1)$ sequences}\label{sec:theory}

Now we present the main theoretical results, including the upper and lower bounds for the parameter $\ell_2$ recovery errors using (\ref{cvx_opt_1}).

%In this section, we focus on bounding the $\ell_2$ norm of estimation error for  \textsc{ar}$(1)$ case. Most importantly, we will show that accurate estimation is only guaranteed for those time series which have piecewise constant dynamic backgrounds with bounded one-step change magnitude. 

\subsection{Main insight: recoverable region}

%In this subsection, we will briefly elaborate on the sufficient conditions for accurate estimation and the intuition behind them. Those conditions correspond to the boundary of recoverable region, which is made up of recoverable instances. 

We start with some necessary definitions.
Define a sequence as being {\it $\varepsilon$-recoverable}, if the $\ell_2$ recovery error using (\ref{cvx_opt_1}) is smaller than $\varepsilon$. The collection of recoverable sequences forms the {\it $\varepsilon$-recoverable region}. We make the following assumptions to ensure the dynamic background does not change too drastically. (i) The background contains at most $s$ changes over the time horizon $T$, i.e., it consists of at most $s+1$ pieces. In other words, the rate-of-change for the dynamic background is on the order of $s/T$, and the change does not happen very often.
(ii) The magnitude of the each change is upper-bounded by $\delta_0$: 
\begin{equation}\label{slowly_varying}
   |\Delta_i| \leq \delta_0, \quad i = 2,\dots,T,
 \end{equation} 
where $\Delta_i = f_i - f_{i-1},\ i = 2,\dots,T$ are one-step changes of the dynamic background.

It is known that the least square estimator is consistent for any stationary autoregressive time-series. Intuitively, the ``smaller'' the dynamic background,  the ``closer'' the sequence is to its stationary counterpart. %This is reflected by  non-stationary time series to its stationary counter-part, the more accurate the least square estimator should be. Since this deviation is caused by the dynamic background and characterized by $s$ and $\delta_0$, it is not surprising that the non-asymptotic bounds on estimation error depend on $s$ and $\delta_0$.
A fundamental question is  \textit{What ranges of dynamic backgrounds and autoregressive coefficient can be estimated accurately?} 
%\begin{equation*}
% %\begin{array}{rl}
%     \textit{What ranges of dynamic background and autoregressive coefficients can be estimated accurately?} \tag{$*$}
%     %\end{array}
% \end{equation*} 
%
We answer this question via Theorems~\ref{thm_main_upper} and \ref{thm_main_lower}, which establish the upper and lower bounds of the recovery error that depend on the number of changes $s$ and the size of the change $\delta_0$. However, in our setting, $s(T)$ is non-decreasing with respect to the time horizon $T$. Thus, we cannot expect the recovery error to shrink to zero with increasing $T$ as the usual asymptotic analysis. 

%We characterize such a set of ``nearly optimal'' recoverable sequences within the $\varepsilon$-recoverable region. 
Theorem~\ref{thm_main_upper} establishes the sufficient condition to ensure  the $\ell_2$ recovery error does not exceed $\varepsilon$: \begin{equation}\label{epsion_detect}
%\begin{array}{rl}
    \min \left\{\Tilde{C}_1  s^{3/2} \delta_0 ,2\sqrt{{\operatorname{vol}(\mathcal{S})}/{\pi}} \right\} +  s^{1/2}\delta_0 \leq \varepsilon - \delta := \varepsilon_\delta,
    %\end{array}
\end{equation}
where $\Tilde{C}_1$ is a positive constant, $\mathcal{S}$ is a user-specified set that the true parameter resides in and $\operatorname{vol}(\cdot) $ denotes the volume of a set. Therefore,  \eqref{epsion_detect} is a sufficient condition that $s$ and $\delta_0$ of a $\varepsilon$-recoverable sequence needs to satisfy, and thus it defines the boundary for $\varepsilon$-recoverable region as illustrated in Figure~\ref{fig:recoverable_illustration}. The recoverable region is the union of the blue and the green regions in Figure~\ref{fig:recoverable_illustration}: the green region does not vary but
the blue region shrinks with increasing $\operatorname{vol}(\mathcal{S})$ and eventually vanishes once $2\sqrt{\operatorname{vol}(\mathcal{S})/\pi}$ exceeds $\varepsilon_\delta$. This can be explained by that when the unknown coefficients reside in a larger $\mathcal{S}$, it is more difficult to recover the true parameters for the same accuracy $\varepsilon$, which leads to a smaller recoverable region. Moreover, as illustrated in Figure \ref{fig:recoverable_illustration}, in the $\varepsilon$-recoverable region, there exists a collection of instances (the region shaded in red dashed lines) where the best achievable performance (lower bound) meets the upper bound of our proposed estimator over the finite time horizon. 

\begin{figure}[H]
\centering
%\vspace{-.1in}
\subfigure{\includegraphics[width=\linewidth]{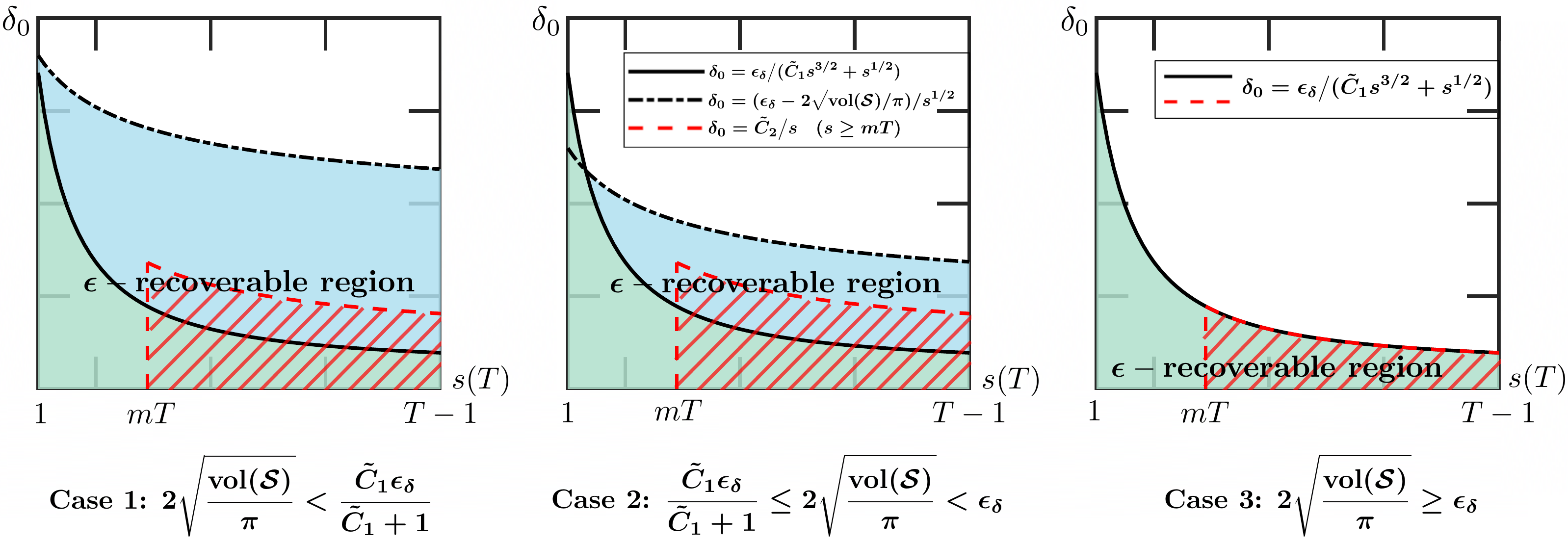}}
    %\vspace{-.15in}
    \caption{Illustration of $\varepsilon$-recoverable region for fixed $T$. This region becomes smaller when hypothesis class $\mathcal{S}$ grows larger. The expressions for the curves in case 1 are the same with case 2. The upper bound \eqref{upper_error_bound} will be nearly tight in the region shaded in red dashed lines.}\label{fig:recoverable_illustration}
    %\vspace{-0.1in}
\end{figure}

\subsection{Preliminaries}

Denote the observation $ x_{i:j} = (x_i,\dots,x_j)^\T$, where the superscript $^\T$ denotes vector/matrix transpose. Given $ x_{1:T}$ and known history $x_0$, we aim to estimate coefficient vector $\beta = (\alpha_1, \mu, \Delta_2,\dots,\Delta_T)^\T \in \RR^{T+1}$, where $\mu = f_1 , \Delta_i = f_i - f_{i-1} , i = 2,\dots,T$. 
Further denote the random noise vector by $ \varepsilon_{1:T} = (\varepsilon_1,\dots,\varepsilon_T)^\T,$ and the random {\it design matrix} by $\XX = ( x_{0:T-1},
L) \in \RR^{T \times (T+1)},$ where $L \in \mathbb R^{T \times T}$ is the lower triangular matrix with non-zeros entries all being ones. For notational  simplicity, we rewrite \eqref{DGP} as\begin{equation}\label{DGP2}
 x_{1:T} = \XX \beta +  \varepsilon_{1:T}.
\end{equation}

% Next, to aid our analysis and make our statements mathematically formal, we define two spaces: one is the space for the unknown true coefficient vector and the other is the space where we want to estimate the coefficients, i.e. the hypothesis class. 

Denote $ \Delta =  (0,0,\Delta_2,\dots,\Delta_T)^\T$ and $\norm{\cdot}_q$ to be $\ell_q$ vector norm. Here, we slightly abuse the notation and assume $\beta$ (instead of $\Delta$) has at most $s$ non-zero entries. Since the total variation constraint only encourages a sparse structure on $\Delta$, $\Delta$ will have at most $s-2$ non-zero entries:\begin{equation*}
    \mathcal{B} = \{ \Delta: \norm{ \Delta}_0 = s-2 \in \{1,\dots,T-1\}, \  \norm{ \Delta}_\infty\leq\delta_0 \}.
\end{equation*}
The space for the unknown true coefficient vector is defined as\begin{equation}\label{para_space}
    \Theta_T = \{\beta: (\alpha_1,\mu) \in \mathcal{S}',  \Delta \in \mathcal{B} \}.
\end{equation}The hypothesis class $\cX$, i.e. the set where we want to estimate the coefficients, is defined as \[\cX = \{\beta: (\alpha_1,\mu) \in \mathcal{S}, \  \norm{ \Delta}_1 < \delta\},\]where
\[\mathcal{S} = \{(\alpha_1,\mu): \alpha_1^2 + \mu^2 \leq \delta_s^2 \}.\]
Here, $\delta_s$ is a user-specified parameter, which specifies the size of hypothesis class $\cX$. We assume $\cX$ contain the ground truth, i.e., $\mathcal{S}' \subset \mathcal{S}$.
% \begin{remark}
% Here, $\delta_s$ is a user-specified parameter and will be set based on prior knowledge. It specifies the range in which we want to search for the autoregressive coefficient and the initial dynamic background, or rather the size of hypothesis class $\cX$. $\cX$ does not necessarily contain the ground truth, but we assume it does in this work, i.e. $\mathcal{S}' \subset \mathcal{S}$. When the true coefficients do not lie in $\cX$, we say this hypothesis class $\cX$ is misspecified. Typically, the ground truth is unknown and we will never know whether $\cX$ is correctly specified. Further discussion on this specification problem is beyond the scope of this work and we will leave it to future discussion.
% \end{remark}

Our goal is to estimate the unknown $\beta \in \Theta_T$ by $\hat \beta_T \in \mathcal{X}$ by solving the following convex optimization problem:
\begin{equation}
    \label{l1_constraint} \hat \beta_T =  \arg \min_{\beta \in \cX} \frac{1}{2T}   \norm{ x_{1:T} - \XX \beta}_2^2.
\end{equation}
Note that this is slightly different from \eqref{cvx_opt_1} (which only has constraint $\norm{\Delta}_1 < \delta$). However, since $\delta_s$ is typically set to be sufficiently large such that the solution will not occur on the boundary of $\mathcal{S}$, \eqref{l1_constraint} leads to the same estimate as \eqref{cvx_opt_1}. %We will see this additional constraint $(\alpha_1,\mu) \in \mathcal{S}$ is just an artifact to aid our analysis.

\subsection{Upper bound}\label{subsec:insight}

Here, we start with a non-asymptotic upper bound for the $\ell_2$ recovery error for the model coefficients and derive the condition for recoverable sequences in \eqref{epsion_detect}.
\begin{theorem}[Upper Bound on $\ell_2$ estimation error] \label{thm_main_upper} For $\hat \beta_T$ defined by \eqref{l1_constraint}, for any $ A_1 >1$, $A_2 > \sqrt{A_1}$ and $A_3 > 0$, and for any selected hyperparameter $\delta$, with probability at least $1- (2T)^{1-A_1} -(2T)^{1-A_2^2/A_1}- 2(2T)^{-A_3^2/A_1^2}$, we have\begin{equation}\label{upper_error_bound}
%\begin{array}{rl}
    \norm{\hat \beta_T - \beta}_2 \leq \min \left\{\Tilde{C}_1 \sqrt{s} \max \left\{s \delta_0 , \delta\right\},2\sqrt{{\operatorname{vol}(\mathcal{S})}/{\pi}} \right\} + \delta + \sqrt{s}\delta_0,
    %\end{array}
\end{equation}where $\Tilde{C}_1$ is a positive constant dependent on $A_1, A_2$ and $A_3$.
\end{theorem}
Note that \eqref{upper_error_bound} implies smaller $s$ and $\delta_0$ will lead to smaller error. To ensure the upper bound is less than a pre-specified $\varepsilon>0$, we choose $\delta$ at most $\varepsilon/(1 + \Tilde{C}_1 \sqrt{T})$, which leads to the condition for $\varepsilon$-recoverable sequences in \eqref{epsion_detect}. %The collection of series which satisfy \eqref{epsion_detect} forms $\varepsilon$-recoverable region, as illustrated in Figure~\ref{fig:recoverable_illustration} above.

\subsection{Lower bound} 

Now we present a lower bound for the $\ell_2$ recovery error using triangle inequality and then improve it via Fano's method.

A naive lower bound can be derived by triangle inequality:$$\norm{\hat \beta_T - \beta}_2 \geq  \norm{\hat \Delta - \Delta }_2\geq \norm{ \Delta}_2 - \norm{\hat \Delta}_2 \geq \norm{ \Delta}_2 - \delta,$$
where $\norm{\hat \Delta}_2 \leq \delta$ due to the total variation constraint. 
However, since we usually do not know $\norm{ \Delta}_2 = O(\sqrt{s}\delta_0)$ {\it a priori}, we cannot set $\delta$ close to $\norm{ \Delta}_2$ to make sure the best achievable performance. Besides, to control the worst performance (i.e. upper bound), $\delta$ is $O(\varepsilon/\sqrt{T})$. Therefore, for series $ x_{1:T}$ with a large $\norm{ \Delta}_2$, we should expect that the recovery error is $\Theta(\sqrt{s}\delta_0)$. 
However, this type of series is typically outside the recoverable region. We are more interested in the lower bound for those instances with much smaller $\norm{ \Delta}_2$. For a certain type of series (which satisfies assumptions \eqref{sparsity} and \eqref{accuracy} below), we can obtain a tighter lower bound by Fano's method as follows:
\begin{theorem}[Lower bound on $\ell_2$ estimation error]\label{thm_main_lower}If there exist $ \Tilde{C}_2>0$ and $0<m<M\leq1$ such that
\begin{align}
    \label{sparsity}
    s(t) \in [mt,Mt], \quad  t = 1,\dots,T_0,\\
    \label{accuracy}
     s(t) \delta_0(t) \leq \Tilde{C}_2, \quad  t =1,\dots,T_0,
\end{align} then for any $C_6 \in (0,1)$ and for any estimator $\Tilde{\beta}_T$, we have \begin{equation}
%\begin{array}{rl}
    \sup_{\beta \in \Theta_T}  \mathrm{pr}\left(\norm{\Tilde{\beta}_T-\beta}_2 \geq C_2 \right) \geq  1 - C_6,
%\end{array}
\end{equation} and \begin{equation}\label{fano_ineq_final}
%\begin{array}{rl}
    C_2 \geq \frac{1}{2} \exp\left\{-\frac{C_3 + C_4 \Tilde{C}_2 M + C_5 \Tilde{C}_2^2 M^2 + (\log 2)/T}{C_6 m}\right\},
    %\end{array}
\end{equation}where $C_3$, $C_4$ and $C_5$ are some positive constants only dependent on $\delta_s$.
\end{theorem}

% \begin{wrapfigure}{r}{0.4\textwidth}
% %\vspace{-0.1in}
% \includegraphics[width=\linewidth]{bound_illustration.pdf}
% %\vspace{-0.2in}
% \caption{Illustration of the naive lower bound and bound derived by Fano's method for $\ell_2$ error under assumptions \eqref{sparsity} and \eqref{accuracy}. Those instances covered by the red dashed lines in Figure~\ref{fig:recoverable_illustration} satisfies \eqref{sparsity} (with $M = 1$) and \eqref{accuracy}, and therefore their trajectories of $\ell_2$ estimation error will stay within the region shaded in green.}\label{fig:bound_illustration}
% %\vspace{-0.1in}
% \end{wrapfigure}

Under assumptions \eqref{sparsity} and \eqref{accuracy}, the naive lower bound will be of order $O(1/\sqrt{T})$ and therefore the lower bound $C_2$ (constant order) will be tighter. 
We denote it by $C_{\rm lower}$. Besides, \eqref{upper_error_bound} ensures the upper bound will be at most $2\sqrt{\operatorname{vol}(\mathcal{S})/\pi}$ plus a $O(1/\sqrt{T})$ term. Thus, we can also obtain a constant order upper bound $C_{\rm upper}$. In this special case, the $\ell_2$ estimation error will stay within a constant order interval $\norm{\hat{\beta}_T-\beta}_2 \in [C_{\rm lower},C_{\rm upper}]$ with constant probability for $t =1,\dots,T_0$. We illustrate this constant order interval in the region shaded in green in Figure~\ref{fig:bound_illustration}. 

\begin{figure}[!htp]
\centering
\subfigure{\includegraphics[scale=0.45]{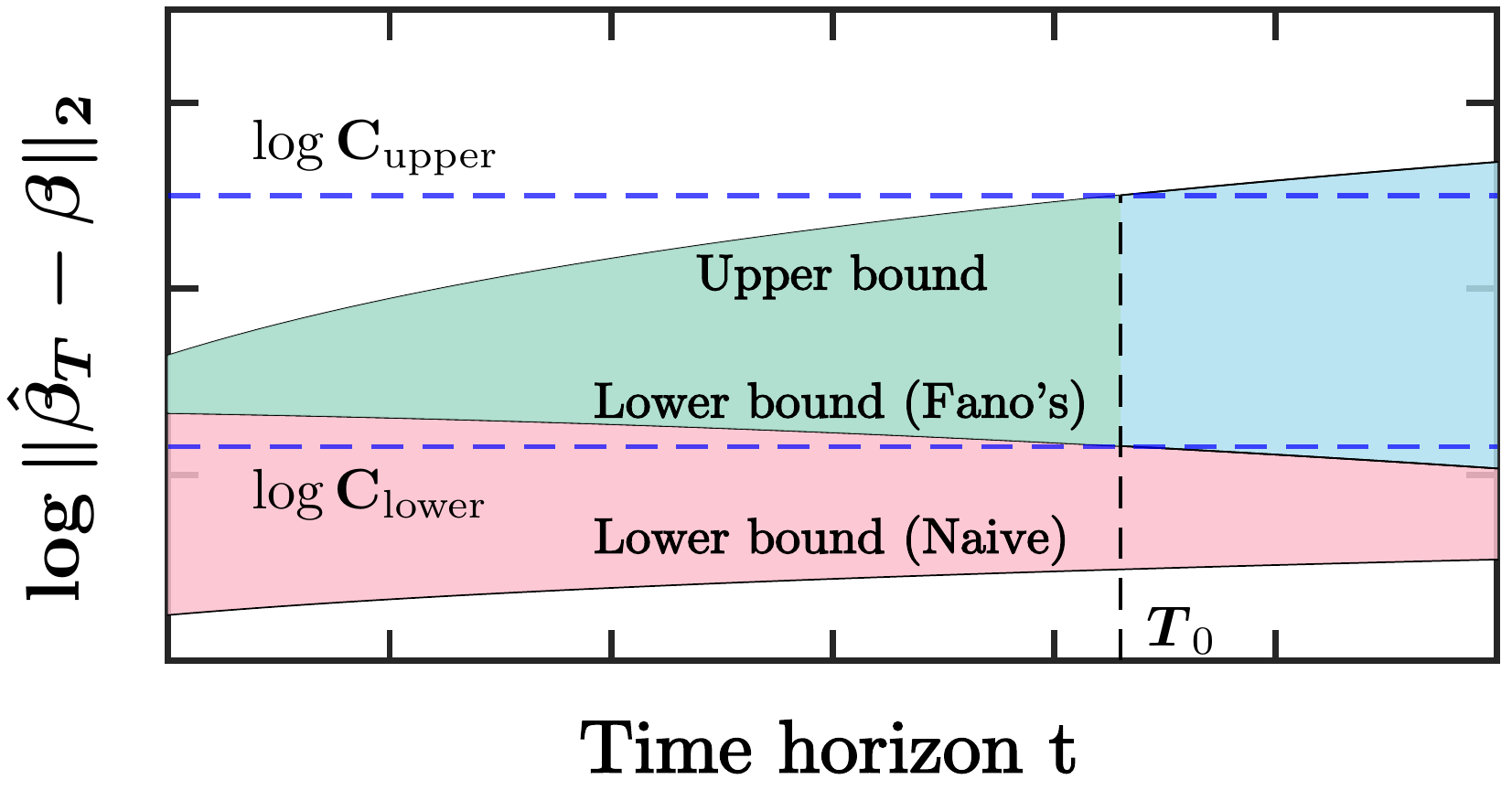}}
    %\vspace{-.1in}
\caption{Illustration of the naive lower bound  the bound derived by Fano's method for $\ell_2$ error under assumptions \eqref{sparsity} and \eqref{accuracy}. The trajectories of $\ell_2$ estimation error for those instances covered by the red dashed lines in Figure~\ref{fig:recoverable_illustration} stay within the region shaded in green. }\label{fig:bound_illustration}
    %%\vspace{-0.1in}
\end{figure}

Figure~\ref{fig:bound_illustration} shows that the upper bound stays close to Fano's lower bound on a finite time horizon, which demonstrates the near-optimality of the proposed method. Besides, the green region in Figure~\ref{fig:bound_illustration} illustrates the constant probability estimation error trajectories for those instances within the red dashed region in Figure~\ref{fig:recoverable_illustration}.

% \begin{figure}[!htp]
% \centering
% \subfigure{\includegraphics[scale=0.45]{fig:bound_illustration.pdf}}
%     %\vspace{-.1in}
%     \caption{Illustration on the trajectory of naive lower bound and bound derived by Fano's method for $\norm{\hat{\beta}_T-\beta}_2$ under assumptions \eqref{sparsity} and \eqref{accuracy}. The region shaded in green is an illustration of the trajectory when our estimation method reaches optimality. }\label{fig:bound_illustration}
%     %%\vspace{-0.1in}
% \end{figure}

\subsection{Proof outline}\label{Proof_outline}

We now present the proof idea for the main theorems. Detailed proofs for Theorems~\ref{thm_main_upper} and ~\ref{thm_main_lower} as well as Propositions~\ref{thm_upper_bound} and ~\ref{thm_lower_bound} are deferred to Appendix~\ref{appendix:proof}.

The proof of Theorem~\ref{thm_main_upper} is largely based on {\it Restricted Eigenvalue} condition and {\it Variational Inequality}. Consider the penalized form of \eqref{l1_constraint}
\begin{equation}
        \hat \beta_T = \label{l1_penal}\arg \min_{\beta \in \RR^{T+1}} \frac{1}{2T} \norm{ x_{1:T} - \XX \beta}_2^2 + \lambda \norm{ \Delta}_1,
\end{equation}
where $\lambda$ is the tuning parameter. By Lagrangian duality, we can show that \eqref{l1_penal} is equivalent to \eqref{l1_constraint}. It is known that \citep{wainwright2019high} there is a one-to-one correspondence between $\delta$ and $\lambda$: if $\hat \beta_T = \hat \beta_T (\lambda)$ minimizes \eqref{l1_penal}, then it also minimizes \eqref{l1_constraint} with $\delta = \norm{\hat{{ \Delta}}}_1$.

The formulation \eqref{l1_penal} links our problem to high-dimensional lasso \citep{wainwright2019high}. This connection motivates us to invoke restricted eigenvalue condition due to \citet{bickel2009simultaneous,van2009conditions} for the design matrix to bound $\ell_2$ estimation error, since the restricted eigenvalue condition is the weakest known sufficient condition according to \citet{raskutti2010restricted}.
Although there have been works verifying restricted eigenvalue conditions for $ x_{0:T-1}$  \citep{loh2011high,basu2015regularized,wu2016performance} or $L$ \citep{harchaoui2010multiple}, they cannot be directly applied for our setting here. Specifically, expanding $x_{0:T-1}$ with a square matrix $L$ in the design matrix leads to the rank-deficiency of $\XX^\T \XX$, thus simply exploring the structure of $ x_{0:T-1}$ cannot address the problem. %On the other hand, since it is relatively easy to verify the restricted eigenvalue condition for full rank $L$, it is not necessary to adopt the more sophisticated analysis in \citet{harchaoui2010multiple} here.

Partition the index set $\{1,\dots,T+1\}$ into three disjoint parts $I_i (i=1,2,3)$, where $I_1 = \{1,2\}$, $I_2$ is the indices for non-zero $\Delta_i$'s and $I_3$ is the indices for zeros in $\beta$. By using an index set $I$ as the subscript of a vector, we keep all entries with indices from $I$ intact and zero out entries with indices from its complement ${I}^{\mathsf{c}}$. 
Since the $\ell_1$ constraint does not encourage sparsity on $\alpha_1$ and $\mu$, we modify the definition of the restricted eigenvalues as follows
\begin{equation}
    \label{phi_minmax} \phi_{\min} (u) = \min_{e \in R_1} \frac{\norm{\XX e}_2}{\sqrt{T} \norm{e}_2}, \quad \phi_{\max} (u) = \max_{e \in R_2} \frac{\norm{\XX e}_2}{\sqrt{T} \norm{e}_2},
\end{equation}where $R_1 = \{e : 2 = \norm{e_{I_1}}_0 \leq \norm{e}_0 \leq u \}$ and $R_2 = \{e : 1 \leq \norm{e}_0 \leq u, \ \norm{e_{I_1}}_0 = 0\}$.

\begin{remark}
The smallest restricted eigenvalue can be understood as follows: take columns of $\XX$ indexed by 1, 2 and another $u-2$ indices from ${I_1}^{\mathsf{c}}$ to form a new matrix $\Tilde{\XX}$, then $\phi_{\min}(u)$ is the smallest one among eigenvalues of all possible $\Tilde{\XX}^\T\Tilde{\XX}/T$'s. Similarly, $\phi_{\max}(u)$ is the largest eigenvalue of $\Tilde{\XX}^\T\Tilde{\XX}/T$, where $\Tilde{\XX}$ is composed of $u$ columns of $\XX$ with indices chosen from ${I_1}^{\mathsf{c}}$. 
\end{remark}
% \begin{remark}
% We can verify $\phi_{\max} (u) \leq u - u(u+1)/2T$ by the fact that for a positive definite (PD) matrix the largest eigenvlaue is upper bounded by its trace.
% \end{remark}

In the following analysis, we use a recently developed technique based on variational inequality \citep{juditsky2019signal,juditsky2020convex} to establish the upper bound.
Consider the gradient field of the objective function in \eqref{l1_constraint}:
\begin{equation*}
    F_{ x_{1:T}}(z) = \big(   A[ x_{1:T}] \ z -  \ a[ x_{1:T}] \big)/T,
\end{equation*}where $A[ x_{1:T}] = \XX^\T \XX$ and $a[ x_{1:T}] = \XX^\T ( \sum_{i=1}^T x_i x_{i-1},   x_{1:T}^\T )^\T$. This vector filed is affine and monotone, since we can verify the symmetric matrix $A[ x_{1:T}]$ is positive semi-definite. The minimizer of \eqref{l1_constraint}, i.e. $\hat \beta_T$, is in fact the solution to the following variational inequality:
\begin{equation*}
   \text { find } z \in \mathcal{X}:\langle  F_{ x_{1:T}}(w), w-z\rangle \geq 0, \quad \forall w \in \mathcal{X}.   \label{VI_1}\tag*{{VI}$[ F_{ x_{1:T}}, \mathcal{X}]$}
\end{equation*}
Moreover, $\beta$ is zero of $\Tilde F_{ x_{1:T}}(z)$ and solution to the following variational inequality:\begin{equation*}
   \text { find } z \in \mathcal{X}:\langle \Tilde F_{ x_{1:T}}(w), w-z\rangle \geq 0, \quad \forall w \in \mathcal{X}, \label{VI_2}\tag*{{VI}$[ \Tilde F_{ x_{1:T}}, \mathcal{X}]$} 
\end{equation*}where \begin{equation*}
   \Tilde F_{ x_{1:T}}(z) =  \big( A[ x_{1:T}] \ z -  \ A[ x_{1:T}] \ \beta \big)/T.
\end{equation*}

We can see that $F_{ x_{1:T}}(z)$ and $\Tilde F_{ x_{1:T}}(z)$ only differ in the following constant term:
$$\eta =  F_{ x_{1:T}}(\beta) - \Tilde F_{ x_{1:T}}(\beta) = F_{ x_{1:T}}(\beta) =  ( A[ x_{1:T}] \ \beta -  \ a[ x_{1:T}])/T.$$ 
Intuitively, the difference between \ref{VI_1} and \ref{VI_2} should reflect the difference between the solutions to those two variational inequalities, i.e. our estimator $\hat \beta_T$ and the ground truth $\beta$. We will show how to bound the $\ell_2$ estimation error using $\norm{\eta}_\infty$ in the following theorem.

% Denote $A  = \mathbb E [\XX^\T] \mathbb E [\XX]$ and $$a = \mathbb E [\XX^\T] \left(\sum_{i=1}^T \mathbb E [x_i x_{i-1}], \mathbb E [x_1],\dots,\mathbb E [x_T] \right)^\T.$$ Consider the gradient field of the expectation of our objective function in \eqref{l1_constraint}
% $$F(\beta) =  ( A \beta -  a)/T:\mathcal X \rightarrow \mathbb R ^ {T+1},$$
% where $\mathcal{X}$ is as in \eqref{l1_constraint}. 

% We can  see this vector filed is affine and monotone\footnote{Vector field $F:\mathcal X \rightarrow \mathbb R ^ {n}$ is said to be monotone if $\langle F(w) - F(z), w-z\rangle \geq 0,  \forall w, z \in \mathcal{X}$, where $\mathcal{X}$ is some convex set. The monotonicity of the gradient field is equivalent to the convexity of the objective function.}, since it is easy to show the symmetric matrix $A$ is positive semi-definite (PSD). Moreover, by $\mathbb E [\varepsilon_i] = 0$ and $\mathbb E [x_i|x_{i-1}] = \alpha_1 x_{i-1} + f_i $, we can  derive that $\beta$ is the zero of $F$ and therefore it is a solution of the following weak {\it variational inequality} (VI):\begin{equation*}
%   \hspace{150pt minus 1fil}\text { find } z \in \mathcal{X}:\langle F(w), w-z\rangle \geq 0, \quad \forall w \in \mathcal{X}. \hspace{110pt minus 1fil} \mathrm{VI}[F, \mathcal{X}]
% \end{equation*}

% Since the vector field is continuous, the solution to weak VI (weak solution) is the same as the solution to strong VI (strong solution).

The following proposition establishes the error bound for the auto-correlation coefficient and the initial dynamic coefficient, combined.
\begin{proposition}[Upper Bound on $\ell_2$ estimation error for $\alpha_1$ and $\mu$] \label{thm_upper_bound} For $\hat \beta_T$ defined by \eqref{l1_constraint}, for $ A_1$, $A_2$, $A_3$ and tuning parameter $\delta$ in Theorem~\ref{thm_main_upper}, there exists $k \in (0,1)$ such that $k\lambda = O\left((\log T)^{3/2}/ T^{1/2}\right)$, with probability at least $1- (2T)^{1-A_1} -(2T)^{1-A_2^2/A_1}- 2(2T)^{-A_3^2/A_1^2}$, we have
% \begin{equation}\label{upper_error_bound}
%     \norm{\hat \beta_T - \beta}_2 \leq \frac{4s}{(1-k)^2\kappa^2} \left( \norm{\eta}_{\infty} +  \kappa s \delta_0 \sqrt{1-\frac{1}{T}}  \right),
% \end{equation}
\begin{equation}\label{upper_error_bound_para_of_interest}
%\begin{array}{rl}
    \sqrt{(\hat \alpha_1 - \alpha_1)^2 + (\hat \mu - \mu)^2}\leq \max \left\{ \frac{4\sqrt{2}}{\kappa^2} \left( \frac{\norm{\eta}_{\infty}}{1-k} +  \kappa C_{\delta,\delta_0,s}\right), \delta + s \delta_0\right\},
    %\end{array}
\end{equation}where $\delta_0$ is the magnitude for one-step changes of dynamic background in \eqref{slowly_varying} and \begin{equation*}
    %\begin{array}{rl}
    \kappa = \sqrt{\phi_{\min} (2)} -  \frac{k\sqrt{2(T-1)}}{(1-k)\sqrt{T}} ,\quad C_{\delta,\delta_0,s} = \frac{2 s \delta_0\sqrt{T-1}}{(1-k)\sqrt{T}} + \sqrt{\frac{(s-2)(T-s+1)}{T}} (s \delta_0+\delta). 
    %\end{array}
\end{equation*}Here, $\phi_{\min} (\cdot)$ is defined in \eqref{phi_minmax}. Moreover, we have $$\norm{\eta}_{\infty} \leq C_0(\log T)^{3/2}/ T^{1/2},$$ where $C_0 = C_0(A_1, A_2, A_3)$ is a positive constant.
\end{proposition}

%The above proposition tells us how to bound the $\ell_2$ norm of the first two entries in the error vector. Then we can obtain \eqref{upper_error_bound} in Theorem~\ref{thm_main_upper} by bounding the remaining entries in the error vector via naive triangle inequality. We do not need a more sophisticated bound here since this naive one will be dominated by the bound in the above theorem anyways. 

Next, we establish the lower bound of the $\ell_2$ estimation error using Fano's method.
\begin{proposition}[Lower bound on $\ell_2$ estimation error]\label{thm_lower_bound}For any estimator $\Tilde{\beta}_T$ and constant $ C_2>0$, we have \begin{equation*}
\begin{split}
%\begin{array}{rl}
    \sup_{\beta \in \Theta_T}  \mathrm{pr}\left(\norm{\Tilde{\beta}_T-\beta}_2 \geq C_2 \right) \geq  1 - \frac{C_3 T + C_4\delta_0(T)\sum_{t=2}^T s(t)+ C_5\delta_0^2(T)\sum_{t=2}^T s^2(t) + \log 2}{s(T) \log(1/2C_2)},
   %\end{array}
\end{split}
\end{equation*}where $C_3$, $C_4$ and $C_5$ are positive constants only dependent on $\delta_s$.
\end{proposition}

We can show that Theorem~\ref{thm_main_lower} follows from the above proposition. The key steps in proving this above proposition are to (i) find a large enough $\varepsilon$-packing of $\Theta_T$ and (ii) upper bound the 
Kullback–Leibler (KL) divergence over this packing. 

\section{Extension to $\text{\textsc{ar}$(p)$}$ sequences}\label{sec:extend}

So far we have been focused on analysis for \textsc{ar}$(1)$ sequences; now we discuss how to extend to general cases. For the \textsc{ar}$(p)$ case, we need to change several terms in \eqref{DGP2} (defined by \textsc{ar}$(1)$): the design matrix becomes $\XX = ( x_{0:T-1}, \dots,  x_{-p+1:T-p} , L) \in \RR^{T \times (T+p)},$ where $L \in \mathbb R^{T \times T}$ remains the lower triangular matrix of ones;
the coefficient vector becomes $\beta = ( \alpha_{1:p}^\T,\mu,\Delta_{2},\dots,\Delta_{T})^\T,$ where $ \alpha_{1:p} = (\alpha_1,\dots,\alpha_p)^\T$. We can solve a similar convex optimization problem as that defined in \eqref{l1_constraint} to estimate the parameters, except that the hypothesis class $\cX$ is defined differently $\cX = \{\beta: ( \alpha_{1:p}^\T,\mu) \in \mathcal{S}_p, \  \norm{ \Delta}_1 < \delta\},$ where $\mathcal{S}_p = \{( \alpha_{1:p}^\T,\mu): \norm{ \alpha_{1:p}}_2^2 + \mu^2 \leq \delta_s^{p+1} \}.$ Moreover, we will redefine $I_1 = \{1,\dots,p+1\}$, while the definitions for $I_2$ and $I_3$ remain the same as defined in Section \ref{Proof_outline}. The restricted eigenvalues are also defined as \eqref{phi_minmax}, except that the error $e$ are restricted to be in $R_1 = \{e: p+1 = \norm{e_{I_1}}_0 \leq \norm{e}_0 \leq u \}$ when calculating $\phi_{\min} (u)$. With these definitions, we can show the following upper bound for the $\ell_2$ recovery error:
\begin{theorem}[Upper Bound on $\ell_2$ estimation error for  \textsc{ar}$(p)$ case] \label{thm_upper_bound_arp} For $\hat \beta_T$ defined by \eqref{l1_constraint} and for all $ A_1 >1$, $A_2 > \sqrt{A_1}$ and $A_3 > 0$, for any selected tuning parameter $\delta$, with probability at least $1- (2T)^{1-A_1} -(2T)^{1-A_2^2/A_1}- 2p(2T)^{-A_3^2/A_1^2}$, we have\begin{equation}\label{upper_error_bound_arp}
%\begin{array}{rl}
    \norm{\hat \beta_T - \beta}_2 \leq \min \left\{\Tilde{C}_3 \sqrt{s} \max \left\{s \delta_0 , \delta\right\},2\sqrt{{\Gamma\left(\frac{p+3}{2}\right)\operatorname{vol}(\mathcal{S}_p)}/{\pi^{\frac{p+1}{2}}}} \right\} + \delta + \sqrt{s}\delta_0,
    %\end{array}
\end{equation}where $\Tilde{C}_3$ is a positive constant dependent on $A_1, A_2$ and $A_3$ and $\Gamma(\cdot)$ is the gamma function.
% \footnote{When non-negative integer p is even, $\Gamma(\frac{p+3}{2}) = (p+\frac{1}{2})(p-\frac{1}{2})\cdots \frac{1}{2} \sqrt{\pi}$; when $p$ is odd, $\Gamma(\frac{p+3}{2}) =(\frac{p+1}{2})!$.}. 
\end{theorem}

Since the expression (\ref{thm_upper_bound_arp}) for the upper bound for  \textsc{ar}$(p)$ case is similar to that in Theorem \ref{thm_main_upper}, the discussion on $\varepsilon$-recoverable region, which is solely determined by the upper bound of estimation error, will be similar too. For lower bounding the estimation error via Fano's method, we can use similar proof strategy as that in Proposition~\ref{thm_lower_bound} or  Lemma~\ref{KL_upper_bound} (although the details are more tedious to specify): (i) express $x_t$ with respect to $\beta, \varepsilon_{1:t}$; (ii) derive the joint distribution of $ x_{1:T}$ based on that expression and (iii) bound the KL divergence. % When $p=1$, step (i) is relatively easy to complete (see \eqref{xt_expression} in Section~\ref{all_proof} in supplementary material) whereas it gets more and more complicated (although not difficult) to do so when $p > 1$.% Thus, we can still follow the proof strategy for \textsc{ar}$(1)$ model. 

\section{Numerical experiments}\label{sec:numerical}

In this section, we perform comprehensive numerical simulations to (i) show that our proposed method works well in practice; (ii) validate our theoretical findings regarding algorithm performance; (iii) compare with existing methods; (iv) demonstrate the good performance of the two proposed bootstrap methods. Recall that our work's primary focus is to estimate the autoregressive coefficients. Thus, we will focus on this in the following four experiments.
% Before moving on to validation of our theoretical results, we first introduce two different but similar estimation methods. The first uses pointwise $\ell_1$ constraint on the one-step change of the slowly varying background instead of the total variation constraint on all one-step change. The estimator is the solution to the following optimization problem
% \begin{equation}\label{cvx_opt}
% %\begin{array}{rl}
% \underset{\beta}{\mbox{minimize}} & \frac{1}{2T} \sum_{i=1}^T (x_i - \sum_{j=1}^p \alpha_j x_{i-j} - f_i)^2\\
% \mbox{subject to} & |f_{i+1} - f_{i}| < \delta, i = 1, \ldots, T-1.
% %\end{array}
% \end{equation}
% where $\delta>0$ is a user-specified parameter. 

% Alternatively, we can also use $\ell_2$ constraints to make the objective differentiable:
% \begin{equation}\label{cvx_opt_2}
% %\begin{array}{rl}
% \underset{\beta}{\mbox{minimize}} & \frac{1}{2T} \sum_{i=1}^T (x_i - \sum_{j=1}^p \alpha_j x_{i-j} - f_i)^2 + \lambda \norm{D_{T-p-1} f}_2^2,
% %\end{array}
% \end{equation}

{\it Experiment 1.} First, we show that our proposed estimation method can accurately recover $\alpha_1$ from non-stationary  \textsc{ar}$(1)$ time series under various settings: $\alpha_1  \in \{0.05, 0.1\}$, $\sigma_0^2 \in \{0.1, 0.2\}$ $\delta_0 \in \{0.05, 0.1\}$ and $T = 5000$. The dynamic background is generated by $f_i = \sum_{k=1}^i \delta_0(U_k-0.5), i = 1,\dots,T$, where $U_k \in [0, 1], k = 1,\dots,T$ is a sequence of i.i.d. uniform random numbers. As discussed above, the accuracy depends on both $s$ and $\delta_0$. Here, we consider an extreme case: $s = T + 1$ (which is supposed to be the most challenging case). Moreover, we also present the results with $\delta$ selected by Durbin-Watson test \citep{durbin1992testing} as an alternative. (Details on the Durbin-Watson test can be found in Section~\ref{tests} in the Appendix~\ref{appendix:background}.) The convex program \eqref{cvx_opt_1} is solved by the $\texttt{cvx}$ package \citep{cvx} and we tune the hyperparameter $\delta$ by Ljung-Box test and Durbin-Watson test, respectively. We repeat the experiment 20 times for each setting, and plot the mean square error of $\hat \alpha_1$, $p$-value of Ljung-Box test and Durbin-Watson test with different $\delta$'s in Figure~\ref{exp2}.

\begin{figure}[htp]
\centering
%\vspace{-0.25in}
\includegraphics[width=1\linewidth]{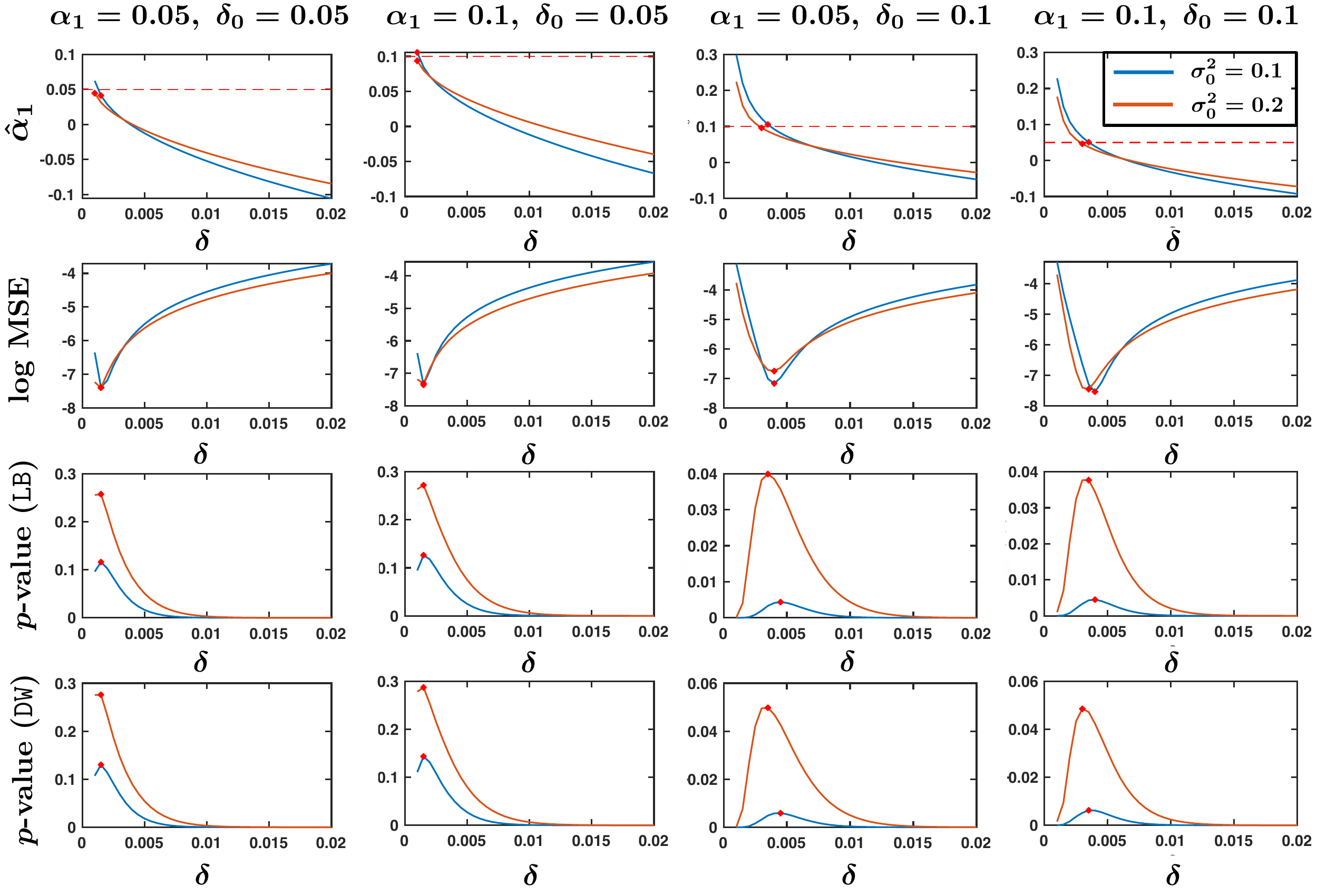}
    %\vspace{-.15in}
    \caption{Performance of our proposed method when $\delta$ increases. The experimental setting is on the top of each column. The red dashed line denotes the ground truth $\alpha_1 = 0.1$. Note that the $\delta$ selected by our proposed tuning procedure (which leads to the maximum $p$-value of Ljung-Box or Durbin-Watson test) gives the best estimate $\hat \alpha_1$.}\label{exp2}
\end{figure}

The results in Figure~\ref{exp2} show that $\hat \alpha_1$ decreases when $\delta$ increases, and
there is a specific value of $\delta$ that leads to the smallest mean square error for estimating $\alpha_1$. In the figure, the red dots in the first two rows correspond to the best achievable $\delta$'s in mean square error. In the last two rows, those red dots indicate the $\delta$'s selected by our proposed tuning procedure. Thus, we can observe that (i) the best achievable $\delta$'s regarding the accuracy and mean square error are roughly the same; (ii) our proposed tuning procedure based on both the Ljung-Box test and Durbin-Watson test can select the best $\delta$. 

\begin{table}[H]
%\vspace{-0.1in}
\caption{Summary of the information of red dots in Figure~\ref{exp2}.}\label{table_alpha}
%\vspace{-0.25in}
\begin{center}
\begin{small}
\begin{sc}
\resizebox{1\textwidth}{!}{%
\begin{tabular}{lcccccccccccccr}
\toprule[1pt]\midrule[0.3pt]
& \multicolumn{3}{c}{\textbf{Average (standard deviation)}}& \multicolumn{3}{c}{\textbf{mean square error}} \\ 
Setting $(\alpha_1,\delta_0,\sigma_0^2)$ & $\varepsilon$-optimal  & Ljung-Box & Durbin-Watson & $\varepsilon$-optimal  & Ljung-Box & Durbin-Watson \\ % Column names row
\cmidrule(l){2-4}
\cmidrule(l){5-7}
 % In-table horizontal line
(0.05, 0.05, 0.10)&4.13$\times 10^{-2}$(2.33$\times 10^{-2}$)&4.13$\times 10^{-2}$(2.33$\times 10^{-2}$)&4.13$\times 10^{-2}$(2.33$\times 10^{-2}$)&6.19$\times 10^{-4}$&6.19$\times 10^{-4}$&6.19$\times 10^{-4}$\\ 
(0.05, 0.05, 0.20)&3.12$\times 10^{-2}$(1.60$\times 10^{-2}$)&3.12$\times 10^{-2}$(1.60$\times 10^{-2}$)&3.12$\times 10^{-2}$(1.60$\times 10^{-2}$)&6.09$\times 10^{-4}$&6.09$\times 10^{-4}$&6.09$\times 10^{-4}$\\ 
(0.05, 0.10, 0.10)&3.91$\times 10^{-2}$(2.03$\times 10^{-2}$)&3.91$\times 10^{-2}$(2.03$\times 10^{-2}$)&5.04$\times 10^{-2}$(2.60$\times 10^{-2}$)&5.33$\times 10^{-4}$&5.33$\times 10^{-4}$&6.77$\times 10^{-4}$\\ 
(0.05, 0.10, 0.20)&3.69$\times 10^{-2}$(2.02$\times 10^{-2}$)&3.69$\times 10^{-2}$(2.02$\times 10^{-2}$)&4.64$\times 10^{-2}$(2.44$\times 10^{-2}$)&5.81$\times 10^{-4}$&5.81$\times 10^{-4}$&6.09$\times 10^{-4}$\\ 
(0.10, 0.05, 0.10)&8.47$\times 10^{-2}$(2.02$\times 10^{-2}$)&8.47$\times 10^{-2}$(2.02$\times 10^{-2}$)&8.47$\times 10^{-2}$(2.02$\times 10^{-2}$)&6.42$\times 10^{-4}$&6.42$\times 10^{-4}$&6.42$\times 10^{-4}$\\ 
(0.10, 0.05, 0.20)&8.01$\times 10^{-2}$(1.65$\times 10^{-2}$)&8.01$\times 10^{-2}$(1.65$\times 10^{-2}$)&8.01$\times 10^{-2}$(1.65$\times 10^{-2}$)&6.68$\times 10^{-4}$&6.68$\times 10^{-4}$&6.68$\times 10^{-4}$\\ 
(0.10, 0.10, 0.10)&9.21$\times 10^{-2}$(2.66$\times 10^{-2}$)&8.14$\times 10^{-2}$(2.41$\times 10^{-2}$)&8.14$\times 10^{-2}$(2.41$\times 10^{-2}$)&7.68$\times 10^{-4}$&9.30$\times 10^{-4}$&9.30$\times 10^{-4}$\\ 
(0.10, 0.10, 0.20)&7.83$\times 10^{-2}$(2.64$\times 10^{-2}$)&8.64$\times 10^{-2}$(3.21$\times 10^{-2}$)&8.64$\times 10^{-2}$(3.21$\times 10^{-2}$)&1.17$\times 10^{-3}$&1.21$\times 10^{-3}$&1.21$\times 10^{-3}$\\ 
\midrule[0.3pt]\bottomrule[1pt]
\end{tabular}
}
\end{sc}
\end{small}
\end{center}
%\vspace{-0.1in}
\end{table}

Table~\ref{table_alpha} summarizes the optimal and the selected $\delta$'s (corresponding to the red dots) in Figure~\ref{exp2}: (i) the average and the standard deviation of $\hat \alpha_1$ obtained by $\varepsilon$-optimal (in the sense of accuracy) $\delta$, $\delta$ selected by Ljung-Box test and Durbin-Watson test and (ii) mean square error of $\hat \alpha_1$ obtained by $\varepsilon$-optimal (in the sense of mean square error) $\delta$, $\delta$ selected by Ljung-Box test and Durbin-Watson test.

{\it Experiment 2.}  Next, we validate our theoretical findings for \textsc{ar}$(1)$ case. The dynamic background is generated in the same way as the previous example. Besides, Figure~\ref{exp2} shows that the $p$-value with respect to $\delta$ is unimodal, which enables us to use the Golden-section search (tolerance $\varepsilon = 0.04$) to tune $\delta$ efficiently. Details on the Golden-section search and this modified tuning procedure can be found in Appendix~\ref{golden}. We also show how the estimate $\hat \alpha_1$ behaves with changing $s$, by setting $\alpha_1 = 0.1, \sigma_0^2 = 0.1$, $\delta_0 = 0.1$ and repeating the same estimation procedure 20 times for each $s \in \{5,35,\dots,3005\}$. The mean and standard deviation of $\hat \alpha_1$ over 20 trials with respect to $s$ in an errorbar plot are plotted Figure~\ref{change}.

\begin{figure}[htp]
%\vspace{-0.1in}
\centering
\subfigure{\includegraphics[width=0.6\linewidth]{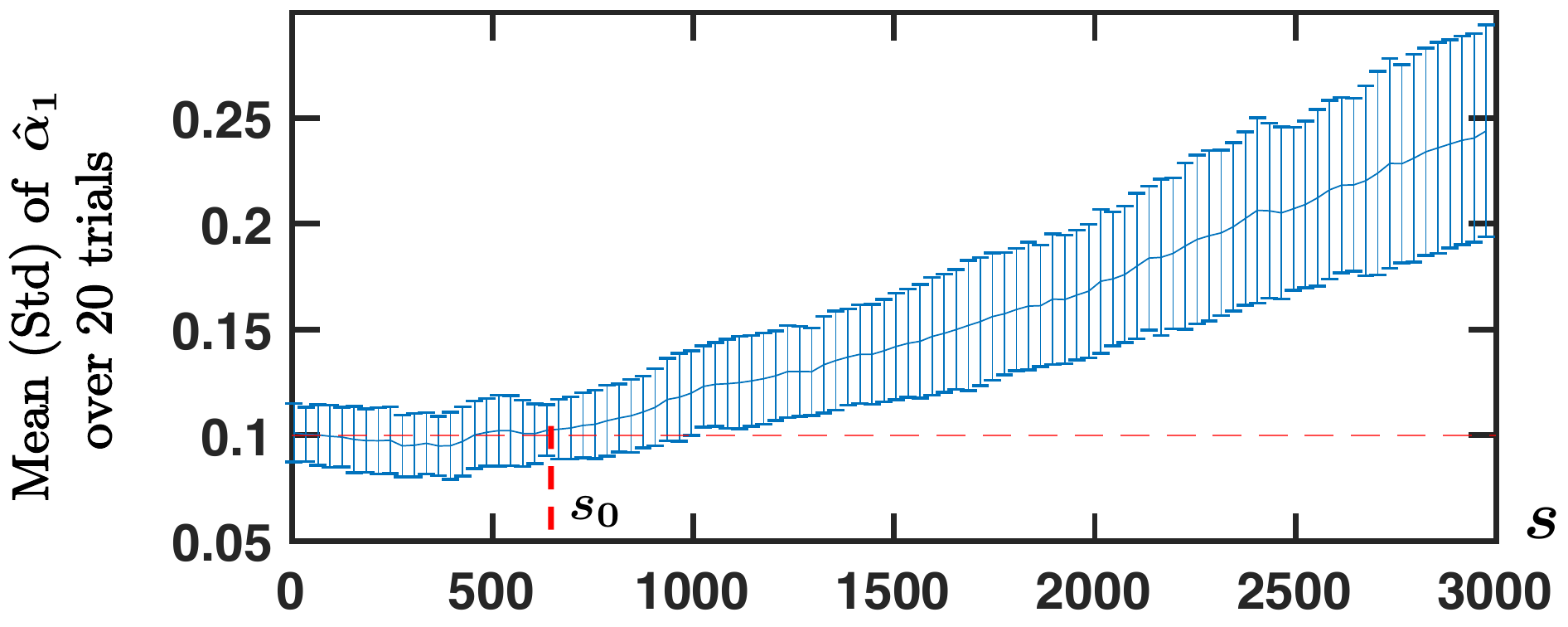}}
    %\vspace{-.2in}
\caption{Algorithmic behavior with respect to $s$. The red dashed horizontal line is the ground truth $\alpha_1 = 0.1$. The estimate starts to deteriorate when $s$ exceeds $s_0 \approx 650$ and becomes worse with larger $s$.}\label{change}    
    %\vspace{-0.1in}
\end{figure}

The results in Figure~\ref{change} show that indeed when $s$ and $\delta_0$ are inside the recoverable region, the estimation error is small, and it will grow with an increasing $s$. Moreover, the error remains small for relatively small $s$, but once $s$ exceeds $s_0$ the error starts to increase; that's when the non-stationary series is not in the $\varepsilon$-recoverable region. 
This observation agrees with our non-asymptotic bounds on estimation error. 
Moreover, we conduct similar experiments for \textsc{ar}$(2)$ case to validate these findings for a more general case; the results can be found in Appendix~\ref{appendix:add_exp}.

{\it Experiment 3.} We compare our method with the method in \citet{zhang2020real}. In the following, we refer to their method as the ``$\ell_2$ variant,'' since it is obtained by solving the convex program with same objective function as \eqref{cvx_opt_1} except for a different constraint: $\sum_{i=1}^{T-1} (f_{i+1} - f_{i})^2 < \delta.$ Again, $\delta\geq0$ is the tuning parameter. We should mention \citet{zhang2020real} did not have a systematic way to tune $\delta$ and here we enhance their method by adding our statistical test based hyperparameter tuning as well.

The piecewise linear dynamic background is generated by $f_i = \sum_{k=1}^i \delta_0(U_k-0.5), i = 1,\dots,T$. Here $U_k = u_i$ for all $k\in \{ k_i, \dots, k_{i+1}\}, i = 0, \dots, s-1$, where $0=k_0<k_1<\cdots<k_s=T$, $k_1,\dots,k_{s-1}$ are randomly selected from $\{1,\dots,T-1\}$ and $u_i \in [0,1], i = 0, \dots, s-1$, is a sequence of i.i.d. uniform random numbers. We consider two cases: (1) $s=1500$, $\delta_0=0.05$, $\norm{ \Delta}_1 = 25.2$; and (2) $s=100$, $\delta_0=0.1$, $\norm{ \Delta}_1 = 46.9$. Here, $s$ denotes the number of changes in the slope; the change vector is not sparse.

\begin{figure}[htp]
%\vspace{-0.1in}
\centering
\subfigure{\includegraphics[width=1\linewidth]{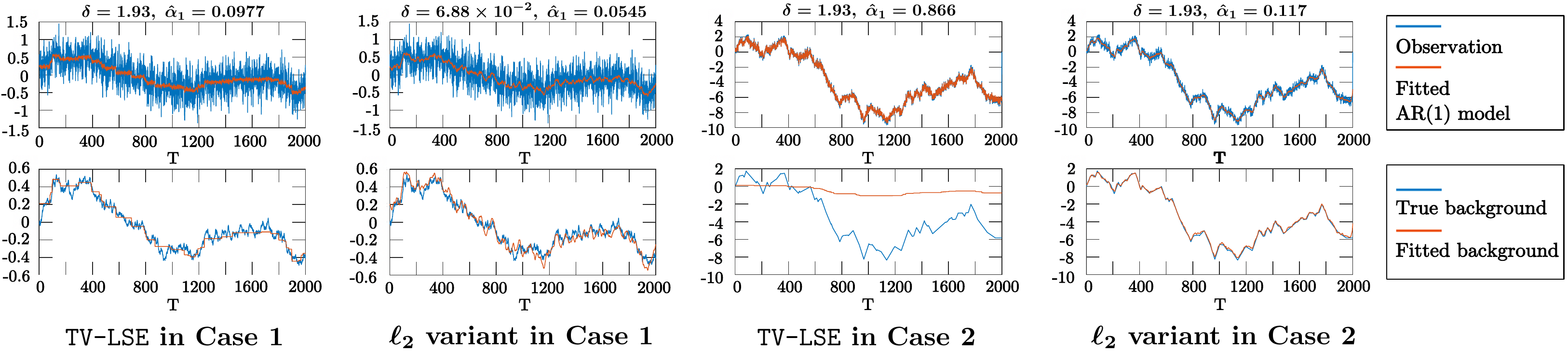}}
    %\vspace{-.15in}
    \caption{Comparison of proposed method and the $\ell_2$ variant in \citet{zhang2020real}. We investigated two cases. In Case 1, the dynamic background oscillates more but a with smaller magnitude, whereas in Case 2, the dynamic background is smoother but has larger one-step difference. Our piecewise constant fitted background better captures the dynamics in Case 1, whereas the $\ell_2$ variant can better approximate the dynamics in Case 2.}\label{exp4}
    %\vspace{-0.15in}
\end{figure}

Figure~\ref{exp4} shows that in Case 1, our proposed method yields a very accurate $\hat \alpha_1$, even though the dynamic background drastically oscillates. This is because the one-step changes are small in magnitude, and therefore, a constant can still serve as a good approximation within some short time window, i.e., this type of sequence is still within the recoverable region. The $\ell_2$ variant yields a  biased estimate for $\alpha_1$, which is probably the reason that \citet{zhang2020real} focus on relatively smooth and structured dynamics. 

In Case 2, even though the dynamic background is smoother than the previous example, the dynamic background changes drastically (large $\norm{ \Delta}_1$). Thus, in this case, the piecewise constant function is a poorer approximation to the dynamic background. This type of sequence is outside the recoverable region, and our proposed method may not work well for those sequences. Nevertheless, the $\ell_2$ variant, together with our proposed hyperparameter tuning procedure, performs well in recovering the serial dependence and serves as an alternative to our proposed estimator. This result agrees with \citet{zhang2020real}, where they demonstrated the good performance of this $\ell_2$ variant when dealing with relatively structured dynamics, 
since $\ell_2$ constraint can lead to a smooth background. In addition, we should mention a polynomial approximation method used in \cite{xu2008bootstrapping} does not perform well in fitting unstructured dynamics and can hardly compete with these two aforementioned non-parametric methods. The numerical comparison with this polynomial method can be found at 
Appendix~\ref{appendix:add_exp}.

% \begin{wrapfigure}{r}{0.45\textwidth}
% %\vspace{-0.15in}
% \includegraphics[width=\linewidth]{bootstrap_illus.pdf}
% %\vspace{-0.25in}
% \caption{Histogram of $\hat \alpha_1$'s from residual-based wild bootstrap samples  (left) and local block bootstrap samples (right). The yellow cycle is estimate from actual observation whereas the red pentagram is ground truth. The blue and red dashed lines are the corresponding boundaries of 90\% and 95\% confidence intervals.}\label{fig:bootstrap}
% %\vspace{-0.15in}
% \end{wrapfigure}

{\it Experiment 4.} Finally, we compare the confidence intervals obtained via two bootstrap methods. We adopt the following experimental setting: $\alpha_1 = 0.1, \sigma_0^2 = 0.1$, $T = 1000$. The dynamic drift is piecewise constant with $\delta_0 = 0.1$, $s = 100$. The bootstrap replication is $N=100$; we use standard normal random numbers as $v_t$'s in residual-based wild bootstrap; for local block bootstrap, we choose block size $b=20$ and local neighborhood length $B = 50$.
We illustrate one replication result by plotting the histogram of $\hat \alpha_1$'s from bootstrap samples in Figure~\ref{fig:bootstrap}.

\begin{figure}[htp]
%\vspace{-0.1in}
\centering
\includegraphics[width=1\linewidth]{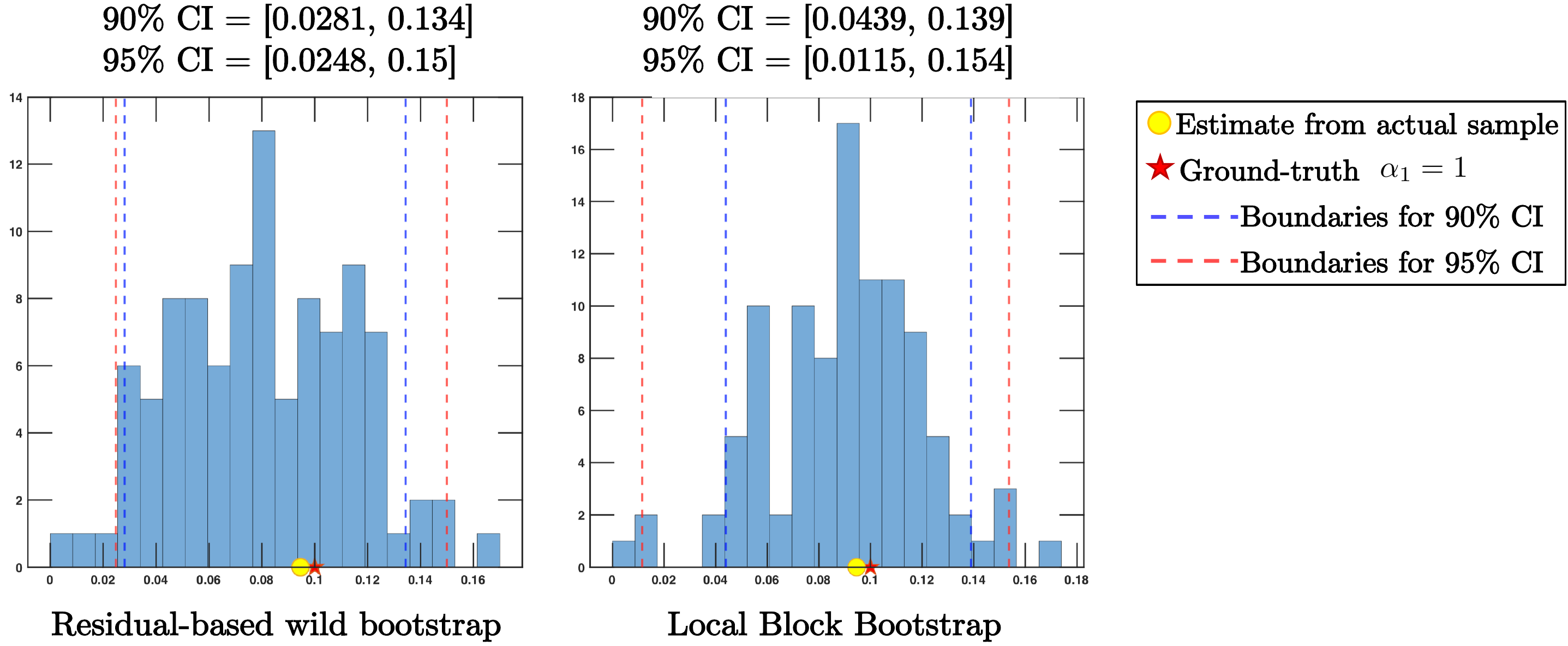}
%\vspace{-0.25in}
\caption{Histogram of $\hat \alpha_1$'s from residual-based wild bootstrap samples  (left) and local block bootstrap samples (right). We can see that the local block bootstrap method yields a smaller confidence interval but lower coverage accuracy.}\label{fig:bootstrap}
    %\vspace{-0.15in}
\end{figure}

From 50 repetition of the above procedure, we find that: (i) the coverage accuracy of 90\% and 95\% confidence intervals: {0.84} and {0.90} for residual-based wild bootstrap and {0.84} and {0.88} for local block bootstrap; 
(ii) the average lengths of 90\% and 95\% confidence intervals: {0.10} and {0.12} for residual-based wild bootstrap and {0.095} and {0.114} for local block bootstrap. The coverage accuracy is slightly lower than the theoretical value since $T = 1000$ is relatively small. The comparison indicates that local block bootstrap tends to yield smaller confidence intervals but has slightly lower coverage accuracy.

\section{Real-data study}\label{sec:real_data}

To validate its performance, we apply our proposed method to real data from a psychological experiment. Consider a reaction time (RT) dataset collected from human subjects. The data are taken from a publicly available database introduced by \citet{Rahnev2020} with 149 individual datasets with human data on different tasks. Here we only analyze a single dataset named \texttt{Maniscalco\_2017\_expt1} chosen based on the fact that it has RT data included and features a large number of trials per subject. 

The data come from an experiment where human subjects made a series of 1000 perceptual judgments over a period of about one hour. Participants were seated in front of a computer and made their responses using a standard keyboard. The task, which is standard in the field, consisted of deciding whether a briefly presented (33 ms) noisy sinusoidal grating was oriented clockwise or counterclockwise from vertical. Subjects responded as quickly as possible but without sacrificing accuracy. The experimenters recorded each judgment's reaction time (that is, the time from the onset of the visual stimulus to the button press used to indicate the subject's response), thus creating a time series of 1000 values for each subject. Data were obtained from 28 subjects.

\begin{figure}[htp]
%\vspace{-0.1in}
\centering
\subfigure{\includegraphics[width=1\linewidth]{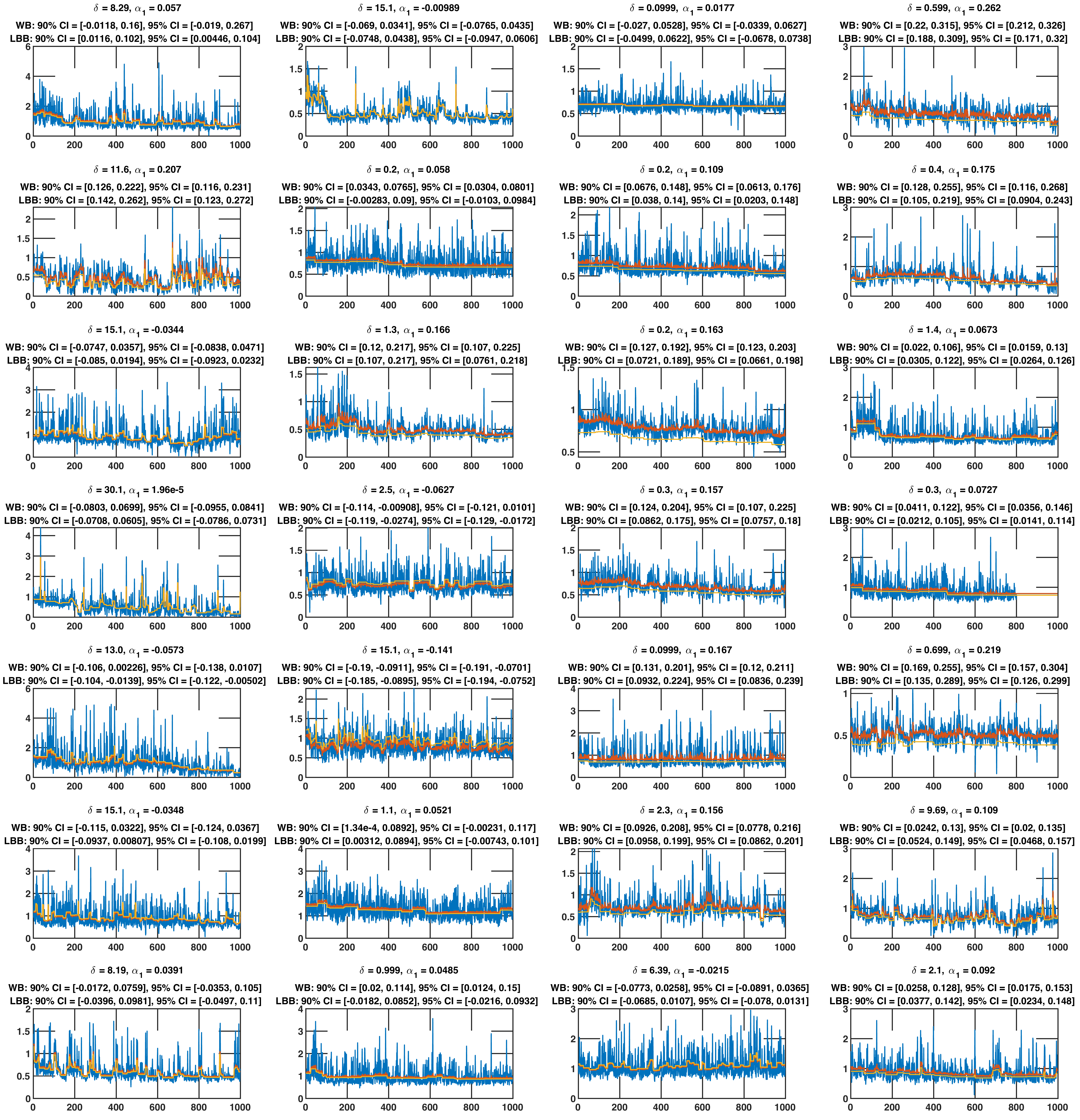}}
    \vspace{-.2in}
    \caption{Experimental results on reaction times for all 28 subjects. Subjects 1 to 28 are organized in the order of left to right and top to bottom. The blue, red and yellow lines correspond to the raw RT values, fitted \textsc{ar}$(1)$ model and fitted dynamic background, respectively. On the top of each figure, we report $\delta$ selected by Ljung-Box test on the logarithm of residuals, estimated \textsc{ar}$(1)$ coefficient and 90\% and 95\% confidence intervals based on residual-based wild bootstrap and local block bootstrap samples. Overall, we observe the presence of substantial drift that varies significantly between subjects but is recovered very well by our proposed method.}\label{fig:realexp_rt}
    %\vspace{-0.1in}
\end{figure}

We first pre-process the raw data by dealing with missing values and obvious outliers. To be precise, we treat RTs that exceed 10 times the interquartile range (i.e., the difference between 75th and 25th percentiles) as outliers and the rest as normal observations. Since naively omitting missing data in time series data will break the serial correlation, we use the median of the normal observations to impute those missing values. The same median is used to replace all outliers. We propose a data-adaptive procedure to tune $\delta$ by applying the Ljung-Box test on the logarithm of original residuals since they are strongly right-skewed. We plot the results for all 28 subjects in Figure~\ref{fig:realexp_rt}. More details on why we choose logarithm transform is in Appendix~\ref{appendix:add_exp}. The confidence intervals are constructed via bootstrapping with the same bootstrapping parameters in our simulation.

\begin{figure}[htp]
%%\vspace{-0.1in}
\centering
\subfigure{\includegraphics[width=\linewidth]{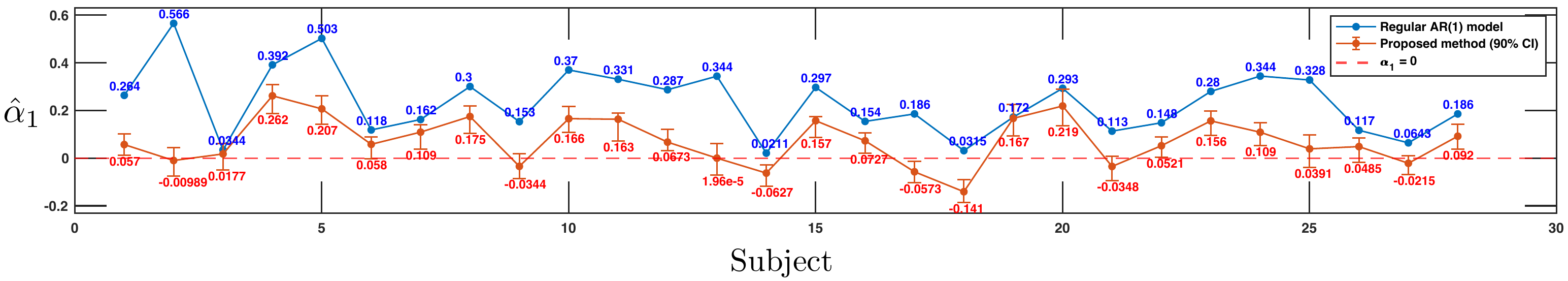}}
    %%\vspace{-.2in}
    \caption{Comparison between our proposed method and regular \textsc{ar}$(1)$ model for all 28 subjects. The errorbar is 90\% confidence interval computed using the local block bootstrap method. Regular \textsc{ar}$(1)$ model typically yields larger $\hat \alpha_1$, since it confuses the dynamic drifts as autoregressive effect.}\label{fig:realexp_rt_compare_1}
    %%\vspace{-0.1in}
\end{figure}

Overall, Figure~\ref{fig:realexp_rt} shows that our method can faithfully capture the underlying dynamics. More specifically, we make four distinct observations:
\begin{itemize}
\item It is clear that there is a substantial background drift. Even in the raw data without any modeling, the drift can often be observed but is even more apparent after applying our method of recovering it. Further, the drift is relatively smooth without big one-step changes, which is exactly the type of dynamical drift that our method can capture well.
\item The background drift has a complex shape and varies significantly from person to person. While for some subjects, the RT series appears to be monotonically decreasing (e.g., subjects 4, 6, 7, 8, 11, 12, 15, 16, 26, and 28) or even close to stationary (e.g., subjects 3 and 19), the remaining subjects exhibit complex trends without any obvious pattern. These differences between subjects demonstrate that the trends need to be identified on the individual time series level and cannot make strong structural assumptions about the dynamic drift. Instead, to be able to capture real data, the dynamic background has to be modeled with minimal structural assumptions. 

\item Our method of fitting the background drift recovers reasonable estimate of the autoregressive coefficient $\alpha_1$. Specifically, $\alpha_1$ is positive or close to 0 for most subjects, which is expected given the extensive previous literature on RT \citep{Laming1968}. Nevertheless, our method recovers a negative $\alpha_1$ for subject 18, which could indicate that the RT series is not universally positively autocorrelated as assumed before and suggests the need for more detailed investigations on this issue. Further, the proposed method appears to provide a good fit for the empirically observed RT data across individual subjects, and the size of the hyperparameter $\delta$ tends to be larger for time series that visually appear to be less stationary. Thus, our method recovers both $\alpha_1$ and $\delta$ well, and provides a useful description of the data dynamics.

\item Our method provides a substantial improvement over the \textsc{ar}$(1)$ model that is typically used to recover the autocorrelation coefficient in psychology and neuroscience. As shown in Figure 9, the \textsc{ar}$(1)$ model leads to very high and clearly inflated estimates of $\alpha_1$ because the model confuses the dynamical drift for an autocorrelation. Overall, our method performs very well on real data from experiments where it is likely to be applied in the future and is a major advance over the standard \textsc{ar}$(1)$ model.

\end{itemize}

\section{Summary}\label{sec:discussion}

In this paper, we develop a total variation constrained least square estimator to estimate serial correlation in the presence of unknown and unstructured dynamic backgrounds. The method approximates the dynamic background via a piece-wise constant function and can approximate a wide range of highly unstructured dynamics. We also developed a statistically principled approach based on the Ljung-Box test to select the tuning parameter. We establish theoretical performance guarantees of our method via upper error bound and develop performance lower bound to characterize the condition for near-optimality (in $\ell_2$ estimation error sense) within the set of $\epsilon$-recoverable sequences. Extensive numerical experiments validate our theory and demonstrate the good performance of our method compared with the state-of-the-art. We apply our method to a psychological study on human reaction times and find that there is indeed substantial and complex-shaped drift in these data. The recovered autocorrelation values are generally positive, confirming the long-hypothesized presence of serial dependence in human reaction times \citep{Laming1968}. The proposed method is thus general and is likely to receive wide adoption in both psychology and neuroscience in studying human and animal decision making.

% \bibliographystyle{biometrika}
% \bibliography{paper-ref}
\section*{Acknowledgement}
The first two authors are supported by NSF CCF-1650913, NSF DMS-1938106, and NSF DMS-1830210. The last author is supported by NIH R01MH119189, NIH R21MH122825, and the Office of Naval Research N00014-20-1-2622.

\bibliographystyle{plainnat}  
\bibliography{paper-ref} %%% Remove comment to use the external .bib file (using bibtex).

\appendix

\newpage

\section{Hyperparameter tuning and Bootstrap confidence interval}\label{appendix:tuning_CI}

Our proposed hyperparameter tuning procedure is:
(i) Set an interval $[\delta_\ell,\delta_u]$ where we believe the best $\delta$ lies in based on prior knowledge; 
(ii) For any $\varepsilon>0$, to make sure the Euclidean distance between selected $\delta$ and the optimal one is less than $\varepsilon$, we divide this interval into $n = \lfloor (\delta_u - \delta_\ell)/\varepsilon \rfloor$ parts with same length $\varepsilon$ and denote the endpoints by $\delta^{(1)},\dots,\delta^{(n+1)}$; 
(iii) For each $\delta^{(j)}$, we fit the proposed estimator as defined by \eqref{cvx_opt_1} and construct the residual sequence by $r_i = x_i - \sum_{j=1}^p \hat \alpha_j x_{i-j} - \hat f_i, \ i = 1,\dots,T$; 
(iv) Apply (Lag-p) Ljung-Box test to the residual sequence to obtain a $p$-value $p_j$; 
(v) The $\varepsilon$-optimal tuning parameter is $\delta^{(j)}$ with $j = \argmax_{j\in\{ 1,\dots,n+1\}} p_j$.
Further details on Ljung-Box test can be found in Section~\ref{tests} in Appendix~\ref{appendix:background}.
% \begin{algorithm}[htp]
%     \caption{Proposed hyperparameter tuning procedure}\label{algo:tuning}
%     \textbf{Input: }Observations $x_{1},\dots,x_T$, given history $x_{-p+1},\dots,x_0$, a pre-specified interval $[\delta_\ell,\delta_u]$ to search the best $\delta$ and tolerance $\varepsilon>0$.
    
%     \textbf{Output: } $\varepsilon$-optimal hyperparameter $\delta$ (by $\varepsilon$-optimality we mean the Euclidean distance between $\delta$ and the optimal one is less than $\varepsilon$).
    
%     \begin{itemize}
%         \item[\textbf{1}] Partition $[\delta_\ell,\delta_u]$ into $n = \lfloor (\delta_u - \delta_\ell)/\varepsilon \rfloor$ parts with length $\varepsilon$ and  endpoints  $\delta^{(1)},\dots,\delta^{(n+1)}$.

%         \item[\textbf{2}] For $j = 1$ to $n+1$, do \begin{itemize}
%             \item[\textbf{(i)}] Fit the $\texttt{TV-LSE}$ as defined by \eqref{cvx_opt_1} with hyperparameter $\delta^{(j)}$ and construct the residual sequence by $r_i = x_i - \sum_{j=1}^p \hat \alpha_j x_{i-j} - \hat f_i, \ i = 1,\dots,T$.
%             \item[\textbf{(ii)}] Apply (Lag-p) Ljung-Box test to the residual sequence $\{r_i\}_{i=1}^T$ to obtain a $p$-value $p_j$.
%         \end{itemize}

%         \item[\textbf{3}] The $\varepsilon$-optimal tuning parameter is $\delta = \delta^{(j)}$ with $j = \argmax_{j\in\{ 1,\dots,n+1\}} p_j$.

% \end{itemize}

% \end{algorithm}

Next, we present how to construct a bootstrap confidence interval.
For our first method residual-based wild bootstrap: (i) we first perform proposed tuning procedure to obtain tuning parameter $\delta$ and the corresponding estimates $\hat \alpha_j$'s and $\hat f_i$'s; (ii) then we calculate the residuals $\hat r_i$'s as suggested in step 2.(i) in proposed tuning procedure; (iii) residual-based wild bootstrap sample is constructed recursively by \eqref{DGP} with $\hat \alpha_j$'s, $\hat f_i$'s and $\Tilde{r}_i = \hat r_i v_i$, where $v_i$'s are i.i.d. random numbers with zero mean and unit variance. As for local block bootstrap, we first choose an integer block size $b$ and local neighborhood size $B$. We partition $T$ samples into $M = \lceil T/b \rceil$ blocks. Then, for $m = 0, \dots, M-1$, the local block bootstrap sample is $\Tilde{x}_{mb+j} = x_{I_m+j-1}, j = 1,\dots,b$, where $I_m$ is a uniform random integer drawn from $\{\max(1,mb-B) ,\dots, \min(T-b+1,mb + B)\}$. In \citet{paparoditis2002local}, it is required that (i) $b/B \rightarrow 0$ as $b \rightarrow \infty$; (ii) when $T \rightarrow \infty$, $T/B \rightarrow 0$ but $B \rightarrow \infty$.

After obtaining the bootstrap sample, we apply proposed tuning procedure to this pseudo-series with $\delta_\ell = \delta - n \varepsilon$ and $\delta_u = \delta + n \varepsilon$ to obtain estimates $\Tilde \alpha_j$'s (we choose $n=2$ in numerical simulation). Then, we repeat this procedure $N$ times to construct a confidence interval by the empirical distribution of $\Tilde \alpha_j$'s. For bootstrap samples, we only need to search around the $\varepsilon$-optimal $\delta$ for the optimal tuning parameter of the pseudo-series since it closely resembles the actual observation. This helps to reduce the computational cost of bootstrapping.

\section{Background knowledge}\label{appendix:background}
\subsection{Ljung-Box test and Durbin-Watson test}\label{tests} Ljung-Box test, sometimes known as the Ljung–Box Q test, is designed to test if there still exhibits serial correlation in the residual sequence. The null hypothesis is $H_0:$ The data are independently distributed. The test statistic is $$Q=T(T+2) \sum_{k=1}^{h} \frac{\hat{\rho}_{k}^{2}}{n-k},
$$where $T$ is the sample size, $\hat{\rho}_{k}$  is the sample autocorrelation at lag $k$, and $h$ is the number of lags being tested. For sequence $\{x_1,\dots,x_T\}$, the sample autocorrelation $\hat{\rho}_{k}$ is defined as$$
\hat{\rho}
_k=\frac{\hat{\gamma}(k)}{\hat{\gamma}(0)}, \ 
\text{  where } \ 
\hat{\gamma}(k)=\frac{1}{T} \sum_{t=1}^{T-|k|}\left(x_{t+|k|}-\bar{x}\right)\left(x_{t}-\bar{x}\right).
$$Here, $\{x_1,\dots,x_T\}$ is residual sequence if one wants to implement Ljung-Box test. Under $H_0$, the test statistic asymptotically follows a $\chi _{{(h)}}^{2}$ distribution. The $p$-value of Ljung-Box test is $\mathrm{pr}(\chi _{{(h)}}^{2} > Q)$.

Durbin-Watson test serves the same purpose. For residual $ e_{t}=\rho e_{t-1}+\nu _{t}$, the test statistic is
$$
d=\frac{\sum_{t=2}^{T}\left(e_{t}-e_{t-1}\right)^{2}}{\sum_{t=1}^{T} e_{t}^{2}}.
$$It tests null hypothesis: $H_0: \rho =0$ against alternative hypothesis $H_1: \rho \neq 0$.

\subsection{Golden-section search}\label{golden}

Golden-section search is a efficient and robust technique for finding an extremum (minimum or maximum) of a function inside a specified interval. For any given $\delta$, if we solve the convex program \eqref{cvx_opt_1}, calculate the residual sequence and perform the hypothesis test on it as we mentioned in Section\ref{subsec:tuning}, we will obtain a $p$-value. That is, we have a mapping that maps $\delta$ to $p$, which we denote as $p=f(\delta)$. In our numerical experiment, we show that $f$ is unimodal by Figure~\ref{exp2}. Therefore, we can speed up the parameter tuning procedure by Golden-section search. The detailed steps are provided below in Algorithm~\ref{algo_golden}.

\begin{algorithm}[htp]
    \caption{Hyperparameter tuning procedure: a Golden-section search variant.}\label{algo_golden}
    \textbf{Input: }Observations $x_{1},\dots,x_T$, given history $x_{-p+1},\dots,x_0$, a pre-specified interval $[\delta_\ell,\delta_u]$ to search the best $\delta$ and tolerance $\varepsilon>0$.
    
    \textbf{Output: } $\varepsilon$-optimal hyperparameter $\delta$.
    \begin{itemize}
        \item[\textbf{1}] Determine two intermediate points $\delta_1 = \delta_\ell + d$ and $\delta_2 = \delta_u - d$, where $d = \frac{\sqrt{5}-1}{2} (\delta_u - \delta_\ell)$
        \item[\textbf{2}] For $k=1,2$: fit the proposed estimator as defined by \eqref{cvx_opt_1} with hyperparameters $\delta_k$; construct the residual sequence by $r_i^{(k)} = x_i - \sum_{j=1}^p \hat \alpha_j{(k)} x_{i-j} - \hat f_i{(k)}, \ i = 1,\dots,T$; apply (Lag-p) Ljung-Box test to the residual sequence $\{r_i{(k)}\}_{i=1}^T$ to obtain a $p$-value $p_k = f(\delta_k)$. 
        
        If $f(\delta_1)>f(\delta_2)$, update $\delta_\ell, \delta_1, \delta_2, \delta_u$ as follows $$
\begin{array}{l}
\delta_\ell=\delta_{2}, \ \delta_{2}=\delta_{1}, \ \delta_{u}=\delta_{u}, \ \delta_{1}=\delta_{1}+\frac{\sqrt{5}-1}{2}\left(\delta_{u}-\delta_\ell\right);
\end{array}
$$Otherwise, update $\delta_\ell, \delta_1, \delta_2, \delta_u$ as follows$$
\begin{array}{l}
\delta_\ell=\delta_\ell, \ \delta_{u}=\delta_{1}, \ \delta_{1}=\delta_{2}, \ \delta_{2}=\delta_{u}-\frac{\sqrt{5}-1}{2}\left(\delta_{u}-\delta_\ell\right).
\end{array}
$$

        \item[\textbf{3}] If $\delta_{u}-\delta_\ell<\varepsilon$, set $\delta_{\max} = (\delta_{u}+\delta_\ell)/2$ and stop iterating; otherwise, go back to step \textbf{2}.    

\end{itemize}
 
\end{algorithm}

Compared to $\lfloor (\delta_u - \delta_\ell)/\varepsilon \rfloor + 1$ searches in proposed tuning procedure, Golden-section search can achieve $\varepsilon$-optimality with just $\lfloor\log(\varepsilon/(\delta_u - \delta_\ell))/\log (0.618)\rfloor + 1$ searches.

\section{Proofs}\label{appendix:proof}

\subsection{Proof of Theorem~\ref{thm_main_upper}}

To begin with, we prove Theorem~\ref{thm_main_upper} by using Proposition~\ref{thm_upper_bound}:

\begin{proof}[Proof of Theorem~\ref{thm_main_upper}]
Denote estimation error by $e = \hat \beta_T - \beta$. By triangle inequality, we have\begin{equation*}
    \sqrt{(\hat \alpha_1 - \alpha_1)^2 + (\hat \mu - \mu)^2} = \norm{e_{I_1}}_2 \leq\sqrt{ 2(\hat \alpha_1^2 +  \alpha_1^2+\hat \mu^2 +  \mu^2) } \leq  2\delta_s  =
2\sqrt{{\operatorname{vol}(\mathcal{S})}/{\pi}}.
\end{equation*}

By definition \eqref{phi_minmax}, $\phi_{min} (2)$ is the smallest eigenvalue of $\Tilde{\XX}^\T \Tilde{\XX}/T$, where $\Tilde{\XX} = (x_{0:T-1}, \bm 1)$ and $\bm 1$ is vector of all ones. Since $\phi_{min} (2) = 0$ if and only if $x_{0:T-1} = a \bm 1$ for some $a \in \RR$, $\phi_{min} (2)$ will be of constant order with overwhelming probability. Since $k$ can be chosen arbitrarily small, $\kappa$ can be lower bounded by a positive constant with high probability. Since $\norm{\eta}_{\infty}= O\left((\log T)^{3/2}/ T^{1/2}\right)$, for large enough $T$, we can simplify Proposition~\ref{thm_upper_bound} into
$$\norm{e_{I_1}}_2 \leq  \Tilde{C}_1 \sqrt{s} \max \left\{s \delta_0 , \delta\right\},$$ 
where $\Tilde{C}_1>0$ is a constant. Together with the naive upper bound by triangle inequality, we obtain $$\norm{e_{I_1}}_2 \leq \min \left\{\Tilde{C}_1 \sqrt{s} \max \left\{s \delta_0 , \delta\right\},2\sqrt{{\operatorname{vol}(\mathcal{S})}/{\pi}} \right\}.$$

Since $\norm{\hat {{ \Delta}}}_2 \leq \norm{\hat {{ \Delta}} }_1 \leq \delta$  and $\norm{ \Delta}_2 \leq \sqrt{s}\delta_0$, by triangle inequality, we have $$\norm{e_{I_2\cup I_3}}_2 = \norm{\hat {{ \Delta}} -  \Delta }_2 \leq \norm{\hat {{ \Delta}}}_2 + \norm{ \Delta}_2 \leq \delta + \sqrt{s}\delta_0.$$

Again, by triangle inequality, $\norm{\hat \beta_T - \beta}_2 \leq \norm{e_{I_1}}_2 + \norm{e_{I_2\cup I_3}}_2$. We complete the proof.
\end{proof}

The proof of Proposition~\ref{thm_upper_bound} is highly involved. We sketch its proof as follows:

\begin{proof}[Proof of Proposition~\ref{thm_upper_bound}]
We first state four very useful lemmas. 
\begin{lemma}[High probability bounds for sub-Gaussian noise]\label{subG_bound}
For sub-Gaussian random noise $\varepsilon_1,\dots,\varepsilon_T \overset{\text{i.i.d.}}{\sim} \operatorname{subG}(\sigma_0^2)$ and $x_1,\dots,x_T$ generated by \eqref{DGP2} (given $x_0$), for all $ A_1 >1$, $A_2 > \sqrt{A_1}$ and $A_3 > 0$, define events $$\cA_1 = \left\{ |\varepsilon_i| \leq \sqrt{2A_1 \sigma_0^2 \log (2T)}, \ i= 1,\dots,T \right\},$$ and $$\cA_2 = \left\{  \left|\sum_{i=j}^T \varepsilon_i\right|\leq  2 A_2 \sigma_0 \sqrt{T}\log(2T), \ j = 1,\dots,T \right\},$$
% \begin{equation*}
% \begin{array}{rl}
%     \cA_1 = \left\{ |\varepsilon_i| \leq \sqrt{2A_1 \sigma_0^2 \log (2T)}, \ i= 1,\dots,T \right\}, \quad  \cA_2 = \left\{  \frac{1}{T}\left|\sum_{i=j}^T \varepsilon_i\right|\leq  2 A_2 \sigma_0 \log(2T)/\sqrt{T}, \ j = 1,\dots,T \right\},
%     \end{array}
% \end{equation*}
we have \begin{equation*}
%\begin{array}{rl}
    \mathrm{pr}\left( \cA_1 \right) \geq 1 - (2T)^{1-A_1}, \quad \mathrm{pr}\left( \cA_2|\cA_1 \right)\geq 1-(2T)^{1-A_2^2/A_1}.
    %\end{array}
\end{equation*}
Furthermore, if we assume there exists a constant $C_1>0$ such that \begin{equation}\label{boundedbackground}
    |f_i| \leq C_1 \sqrt{ \log T}, \quad i=1,\dots,T,
\end{equation}define event $$\cA_3  = \left\{\left|\sum_{i=1}^T \varepsilon_i x_{i-1} \right| \leq  2\sqrt{2} A_3 \sigma_0^2(c_1+1) \log(2T)\sqrt{T\log(2T)}/({1-\alpha_1})\right\},$$
% \begin{equation*}
% %\begin{array}{rl}
%     \cA_3  = \left\{ \frac{1}{T}\left|\sum_{i=1}^T \varepsilon_i x_{i-1} \right| \leq  \frac{2\sqrt{2} A_3 \sigma_0^2(c_1+1)}{1-\alpha_1} \log(2T)\sqrt{\frac{\log(2T)}{T}}\right\},
%     %\end{array}
% \end{equation*}
where $c_1>0$ is a constant such that $|f_i| \leq c_1 \sqrt{2A_1 \sigma_0^2 \log (2T)}, i=1,\dots,T$, then we will have\begin{equation*}
    \mathrm{pr}\left( \cA_3| \cA_1\right)  \geq 1 - 2(2T)^{-A_3^2/A_1^2}.
\end{equation*}
\end{lemma}
By Lemma~\ref{subG_bound}, we will have
\begin{equation*}
    \begin{split}
        \mathrm{pr}(\cA_1 \cap \cA_2 \cap \cA_3) &= \mathrm{pr}(\cA_1)\left(1 - \mathrm{pr}(\cA_2^{\mathsf{c}} \cup \cA_3^{\mathsf{c}}|\cA_1)\right)\\
        & >1- (2T)^{1-A_1} -(2T)^{1-A_2^2/A_1}- 2(2T)^{-A_3^2/A_1^2}.
    \end{split}
\end{equation*}
This means event $\cA = \cA_1 \cap \cA_2 \cap \cA_3$ holds with probability at least $1- (2T)^{1-A_1} -(2T)^{1-A_2^2/A_1}- 2(2T)^{-A_3^2/A_1^2}$. 
% &\geq \left(1 - (2T)^{1-A_1}\right) \left(1-(2T)^{1-A_2^2/A_1}- 2(2T)^{-A_3^2/A_1^2}\right)\\

\begin{lemma}[Restricted $\ell_1$ estimation error]\label{restricted_error}
Under assumption \eqref{boundedbackground}, for our proposed estimator $\hat \beta_T$, as defined in \eqref{l1_constraint} or equivalently \eqref{l1_penal}, if we choose $k \in (0,1)$ and tuning parameter $\lambda$ such that $k\lambda = O\left((\log T)^{3/2}/ T^{1/2}\right)$, then on event $\cA$, the estimation error $e = \hat \beta_T - \beta$ satisfies: $$\norm{e_{I_3}}_1 \leq \min \left\{\frac{1+k}{1-k}\norm{e_{I_2}}_1,\frac{2s\delta_0}{1-k}\right\} + \frac{k}{1-k}\norm{e_{I_1}}_1.$$
\end{lemma}

\begin{lemma}\label{lower_bound_lma} Under assumption \eqref{boundedbackground}, on event $\cA$, for any integer $m \leq T+1-|J_0|$,we have\begin{equation}\label{lower_bound}
    \begin{split}
    %\begin{array}{rl}
        \frac{1}{\sqrt{T}} &\norm{\XX e}_2
         \geq  \left(\sqrt{\phi_{min} (2)} - \sqrt{2\phi_{\max} (m)} \frac{k}{1-k}\right) \norm{e_{I_1}}_2 \\
         &  -\left(\frac{2 \sqrt{\phi_{\max} (m)}}{1-k} +\sqrt{(s-2)\left(1-\frac{s-1}{T}\right)}\right) s \delta_0 - \sqrt{(s-2)\left(1-\frac{s-1}{T}\right)} \delta,
         %\end{array}
    \end{split}
\end{equation}where $\phi_{min} (\cdot)$ and $\phi_{\max}(\cdot)$ are defined in \eqref{phi_minmax}.
\end{lemma}

For simplicity, in the following we denote $\kappa = \sqrt{\phi_{min} (2)} - k\sqrt{2\phi_{\max} (m)} /({1-k})$ and \begin{equation*}
%\begin{array}{rl}
C(\delta,\delta_0,s,m) = \frac{2 \sqrt{\phi_{\max} (m)}s \delta_0}{1-k}  + \sqrt{(s-2)\left(1-\frac{s-1}{T}\right)} (s \delta_0+\delta).
        %\end{array}
\end{equation*}We further denote $J_0 = I_1 \cup I_2$, i.e. the set of indices for all non-zero coefficients. 
% By our previous notion, $s$-sparse means the cardinality of the index set of non-zero coefficients is $|J_0| = s$.
% \begin{lemma}\label{lower_bound_lma} Under assumption \eqref{boundedbackground}, on event $\cA$, we have\begin{equation}\label{lower_bound}
%     \begin{split}
%         \frac{1}{\sqrt{T}}\norm{\XX e}_2 \geq 
%           \kappa \norm{e_{J_0}}_2 -  \frac{2s\delta_0 \sqrt{\phi_{\max} (m)}}{1-k},
%     \end{split}
% \end{equation}where $\kappa= \sqrt{\phi_{min} (s)} -  \frac{k}{1-k} \sqrt{2\phi_{\max}(m)}$ and $\phi_{min} (s)$ and $\phi_{\max}(m)$ are defined in \eqref{phi_min} and \eqref{phi_max}.
% \end{lemma}

\begin{lemma}\label{upper_bound_lma}Under assumption \eqref{boundedbackground}, on event $\cA$, we have
\begin{equation}\label{upper_bound}
%\begin{array}{rl}
    \norm{\XX e}_2^2/T \leq 2 \norm{\eta}_{\infty} \norm{e_{J_0}}_1/(1-k),
    %\end{array}
\end{equation}Additionally, we have $\norm{\eta}_{\infty}= O\left((\log T)^{3/2}/ T^{1/2}\right)$.
\end{lemma}

Here, we consider two cases: (i) $\norm{e_{I_1}}_1 \leq \norm{e_{I_2}}_1$ and (ii) $\norm{e_{I_1}}_1 > \norm{e_{I_2}}_1$. In case (i), we have $$\norm{e_{I_1}}_2 \leq \norm{e_{I_1}}_1 \leq \norm{e_{I_2}}_1 = \norm{\hat{{\Delta}}_{I_2} -  \Delta_{I_2}}_1 \leq \norm{\hat{{\Delta}}_{I_2}}_1 + \norm{ \Delta}_1 \leq \delta + s \delta_0.$$ In case (ii), $\norm{e_{J_0}}_1 = \norm{e_{I_1}}_1 +\norm{e_{I_2}}_1 < 2 \norm{e_{I_1}}_1.$ By \eqref{lower_bound} and \eqref{upper_bound}, we have
\begin{equation*}
    \begin{split}
        \frac{2}{1-k} \norm{\eta}_{\infty} \norm{e_{J_0}}_1 &\geq \left(\kappa \norm{e_{I_1}}_2 -  C(\delta,\delta_0,s,m)\right)^2\\
        &\geq \kappa^2 \norm{e_{I_1}}_2 \norm{e_{I_1}}_1/\sqrt{2} -  2 \kappa C(\delta,\delta_0,s,m) \norm{e_{I_1}}_2\\
        &\geq \kappa^2 \norm{e_{I_1}}_2 \norm{e_{J_0}}_1/2\sqrt{2} -  2 \kappa C(\delta,\delta_0,s,m) \norm{e_{J_0}}_1.
    \end{split}
\end{equation*}Rearranging the terms in the above inequality and choosing $m=1$, we have
\begin{equation*}
%\begin{array}{rl}
    \norm{e_{I_1}}_2 \leq \frac{4\sqrt{2}}{\kappa^2} \left( \frac{\norm{\eta}_{\infty}}{1-k} +  \kappa C(\delta,\delta_0,s,1)
 \right).
%\end{array}
\end{equation*}
Denote $C_{\delta,\delta_0,s}=C(\delta,\delta_0,s,1)$. Since $\phi_{\max}(1) = 1-1/T$, combing the results above proves \eqref{upper_error_bound_para_of_interest}. 
% Taking both cases into account, we get
% \begin{equation*}
% %\begin{array}{rl}
%     \norm{e_{I_1}}_2 \leq \max \left\{ \frac{4\sqrt{2}}{\kappa^2} \left( \frac{\norm{\eta}_{\infty}}{1-k} +  \kappa C(\delta,\delta_0,s,m)\right), \delta + s \delta_0\right\}.
%     %\end{array}
% \end{equation*}
\end{proof}

\subsection*{Proofs of Lemmas in the proof of Proposition~\ref{thm_upper_bound}}
\begin{proof}[Proof of Lemma~\ref{subG_bound}]

% Recall that by proof of Lemma~\ref{lem_diff_bound}, we have shown that with probability at least $1-\varepsilon$, $|\varepsilon_i|'s$ are all upper bounded by $\sqrt{2\sigma_0^2 \log(2T)}$. Under sparsity assumption \eqref{sparsity}, the assumption on dynamic background in Lemma~\ref{lem_diff_bound} holds trivially. Similar to the proof in Lemma~\ref{lem_diff_bound}, by the convergence of geometric series, we have $x_i$'s also bounded by $ c_5 \sqrt{2\sigma_0^2 \log(2T)}$, where $c_5$ is some positive constant dependent on $\alpha_1$ and $\mu$.

%  we have 
% $$\mathrm{pr} \left(\frac{1}{T} \bigg|\sum_{i=1}^T x_{i-1} \varepsilon_i \bigg|\geq k \lambda\right) \leq 2 \exp \left\{-\frac{ T k^2 \lambda^2}{8 c_5^2 \sigma_0^4 (\log(2T))^2}\right\}.$$
% If we want this probability smaller than $\varepsilon$, we need \begin{equation}\label{lambda_order_1}
%     \lambda \geq \frac{2\sqrt{2}c_5 \sigma_0^2}{k} \log\left(\frac{2T}{\varepsilon}\right)\sqrt{\frac{\log(2)}{T}}.
% \end{equation}

For sub-Gaussian random noise $\varepsilon_i \sim \text{subG}(\sigma_0^2)$, we will have: $$\mathrm{pr}(|\varepsilon_i| \leq c_2 , \ i = 1,\dots,T) \geq 1 - 2 T \exp\left\{-\frac{c_2^2}{2\sigma_0^2}\right\}.$$ Setting $c_2 = \sqrt{2 A_1 \sigma_0^2 \log(2T)}$, we prove the first inequality.

By the uniform upper bound on the dynamic background \eqref{boundedbackground}, we can find a constant $c_1$ such that dynamic background is uniformly bounded by $c_1 c_2$. Thus, on event $\cA_1$ we will get
$$-(c_1+1)c_2\leq x_i - \alpha_1 x_{i-1} \leq (c_1+1)c_2, \quad (i = 1,\dots,T).$$

By the convergence of geometric series we have $|x_i| \leq (c_1+1)c_2 /(1-\alpha_1)$ and thus we have $$|x_{i-1}\varepsilon_i | \leq \frac{(c_1+1)c_2^2}{1-\alpha_1} = c_3, \quad (i = 1,\dots,T).$$

Since $E    [x_{i-1}\varepsilon_i|x_{i-1}] = x_{i-1} E    [\varepsilon_i] = 0$ and $$\Var (x_{i-1}\varepsilon_i |x_{i-1}) = x_{i-1}^2 \sigma_0^2 \leq \left(\frac{(c_1+1)c_2}{1-\alpha_1}\right)^2\sigma_0^2, $$ $\{x_{i-1}\varepsilon_i\}$ is a bounded martingale difference sequence with respect to filtration $\{ \sigma(x_1,\dots,x_{i-1})\}$.

By Azuma–Hoeffding inequality, we have 
$$\mathrm{pr} \left(\frac{1}{T} \bigg|\sum_{i=1}^T x_{i-1}\varepsilon_i \bigg|\geq c_4\right) \leq 2 \exp \left\{-\frac{ T c_4^2}{2c_3^2}\right\}.$$

Set\begin{equation}\label{Deltabound1}
    c_4 = A_3 \frac{2\sqrt{2}\sigma_0^2(c_1+1)}{1-\alpha_1} \log (2T) \sqrt{\frac{\log (2T)}{T}},
\end{equation}we prove the third inequality.

Similarly, on event $\cA_1$, by Azuma–Hoeffding inequality, we will obtain  \begin{equation*}
    \begin{split}
        \mathrm{pr} \left(\frac{1}{T} \bigg|\sum_{i=j}^T \varepsilon_i \bigg|\geq c_5\right) & \leq 2 \exp \left\{-\frac{ T^2 c_5^2}{2 (T-j) c_2^2}\right\} \\&< 2 \exp \left\{-\frac{ T c_5^2}{4 A_1 \sigma_0^2 \log(2T)}\right\}, \quad j = 1,\dots,T.
    \end{split}
\end{equation*}

Therefore,\begin{equation*}
    \begin{split}
        \mathrm{pr} \left(\frac{1}{T} \bigg|\sum_{i=j}^T \varepsilon_i \bigg| < c_5, \ j = 1,\dots,T\right) &\geq
        1 - \sum_{j=2}^T \mathrm{pr} \left(\frac{1}{T} \bigg|\sum_{i=j}^T \varepsilon_i \bigg|\geq c_5\right) \\&>
        1 - 2 T \exp \left\{-\frac{ T c_5^2}{4 A_1 \sigma_0^2 \log(2T)}\right\},
    \end{split}
\end{equation*}where the first inequality comes from union bound. Again, set \begin{equation}\label{Deltabound2}
    c_5  =  \frac{2 A_2 \sigma_0 \log(2T)}{\sqrt{T}},
\end{equation}we prove the second inequality.
\end{proof}

\begin{proof}[Proof of Lemma~\ref{restricted_error}]
By definition \eqref{l1_penal}, we have
$$\frac{1}{2T} \norm{ x_{1:T} - \XX \hat \beta_T}_2^2 + \lambda \norm{\hat{{ \Delta}}}_1 \leq \frac{1}{2T} \norm{ x_{1:T} - \XX \beta}_2^2 + \lambda \norm{ \Delta}_1.$$

Rearrange terms and we will get
$$\frac{1}{2T} \norm{\XX (\hat \beta_T - \beta)}_2^2 \leq \lambda (\norm{ \Delta}_1 - \norm{\hat{{ \Delta}}}_1) + \frac{1}{T}  \varepsilon_{1:T}^\T \XX (\hat \beta_T - \beta).$$

If we choose $k \lambda$ as follows
\begin{equation*}
    k \lambda = \frac{2\sqrt{2} A_3 \sigma_0^2(c_1+1)}{1-\alpha_1} \log(2T)\sqrt{\frac{\log(2T)}{T}}= O\left(\frac{(\log T)^{3/2}}{ T^{1/2}}\right),
\end{equation*}we have $k \lambda = c_4 > c_5$ for $T$ large enough, where $c_4$ and $c_5$ are defined in \eqref{Deltabound1} and \eqref{Deltabound2}, respectively. Then on event $\cA$, we have \begin{equation} \label{Delta_bound}
     \frac{1}{T} | \varepsilon_{1:T}^\T \XX |= \frac{1}{T}  \left( \bigg |\sum_{i=1}^T \varepsilon_{i} x_{i-1} \bigg |, \bigg |\sum_{i=1}^T \varepsilon_{i}\bigg |, \bigg |\sum_{i=2}^T \varepsilon_{i}\bigg |,\dots,  |\varepsilon_{T} |\right)^\T \leq k \lambda  \mathbf{1},
\end{equation}where $\mathbf{1} \in \RR^T$ is the vector of ones. Thus, we will obtain $ \norm{\eta}_\infty  = \norm{ \varepsilon_{1:T}^\T \XX/T}_\infty \leq k\lambda$ and
$$\frac{1}{2T} \norm{\XX (\hat \beta_T - \beta)}_2^2 \leq \lambda (\norm{ \Delta}_1 - \norm{\hat{{ \Delta}}}_1) + k\lambda \norm{\hat \beta_T - \beta}_1.$$

By pulsing $\lambda (1-k) \norm{e_{I_2 \cup I_3}}_1$ on both side of this equation, we will get
\begin{equation}\label{lma_res_eq1}
    (1-k) \norm{e_{I_2 \cup I_3}}_1 \leq  (\norm{ \Delta}_1 - \norm{\hat{{ \Delta}}}_1 + \norm{e_{I_2 \cup I_3}}_1) + k \norm{e_{I_1}}_1.
\end{equation}

Since $ \Delta = \beta_{I_2 \cup I_3}$, $e_{I_2 \cup I_3} = \hat{{\Delta}} -  \Delta$. By the sparse structure we know \begin{equation}\label{lma_res_eq2}
    \norm{ \Delta}_1 - \norm{\hat{{ \Delta}}}_1 + \norm{e_{I_2 \cup I_3}}_1\leq  2\norm{ \Delta}_1 \leq 2 s \delta_0.
\end{equation}

Meanwhile, since $\norm{ \Delta}_1$ takes value zero on index set $I_3$, we have $\hat \Delta_{I_3} = e_{ I_3}$ and thus $\norm{\hat{{\Delta}}_{I_3}}_1 = \norm{e_{ I_3}}_1$. Therefore, we have \begin{equation}\label{lma_res_eq3}
    \norm{ \Delta}_1 - \norm{\hat{{ \Delta}}}_1 + \norm{e_{I_2 \cup I_3}}_1 = \norm{ \Delta_{I_2 }}_1 - \norm{\hat {{\Delta}}_{I_2 }}_1 + \norm{e_{I_2 }}_1 \leq 2 \norm{e_{I_2 }}_1.
\end{equation}

Plugging \eqref{lma_res_eq2} and \eqref{lma_res_eq3} back into \eqref{lma_res_eq1}, we will get$$ \norm{e_{ I_3}}_1 \leq \min \left\{ \frac{1+k}{1-k}\norm{e_{ I_2}}_1, \frac{2s \delta_0}{1-k}  \right\} + k \norm{e_{I_1}}_1.$$
We complete the proof.
\end{proof}

\begin{proof}[Proof of Lemma~\ref{lower_bound_lma}] By \eqref{l1_constraint}, we have $$\norm{e_{I_2}}_2 \leq \norm{e_{I_2}}_1 = \norm{\hat{{\Delta}}_{I_2} -  \Delta_{I_2}}_1 \leq \norm{\hat{{\Delta}}_{I_2}}_1 + \norm{ \Delta}_1 \leq \delta + s \delta_0.$$ 
 
Partition index set ${J_0}^{\mathsf{c}}$ into L disjoint sets: ${J_0}^{\mathsf{c}} = \cup_{\ell=1}^L J_\ell$, where $|J_1| = \cdots = |J_{L-1}| = m$ and $|J_L|\leq m$, and  $\sum_{\ell=1}^L  \norm{e_{J_\ell}}_2 \leq \sum_{\ell=1}^L  \norm{e_{J_\ell}}_1 = \norm{e_{{J_0}^{\mathsf{c}}}}_1$, we get
\begin{equation*}
    \begin{split}
        \frac{1}{\sqrt{T}}\norm{\XX e}_2 &\geq \frac{1}{\sqrt{T}}\norm{\XX e_{J_0}}_2 - \frac{1}{\sqrt{T}}\norm{\XX e_{{J_0}^{\mathsf{c}}}}_2 \\
        &\geq \sqrt{\phi_{min} (2)} \norm{e_{I_1}}_2 - \sqrt{\phi_{\max} (s-2)} \norm{e_{I_2}}_2 - \sqrt{\phi_{\max} (m)} \sum_{\ell=1}^L  \norm{e_{J_\ell}}_2\\
        &\geq  \sqrt{\phi_{min} (2)} \norm{e_{I_1}}_2 - \sqrt{(s-2)\left(1-\frac{s-1}{T}\right)} (\delta + s \delta_0)- \sqrt{\phi_{\max} (m)}  \norm{e_{{J_0}^{\mathsf{c}}}}_1.
    \end{split}
\end{equation*}

Since $I_3 = {J_0}^{\mathsf{c}}$, $\sqrt{2}\norm{e_{I_1}}_2 \geq \norm{e_{I_1}}_1$, by Lemma~\ref{restricted_error}, we have
\begin{equation*}
    \begin{split}
        \frac{1}{\sqrt{T}} & \norm{\XX e}_2
         \geq  \left(\sqrt{\phi_{min} (2)} - \sqrt{\phi_{\max} (m)} \frac{\sqrt{2}k}{1-k}\right) \norm{e_{I_1}}_2 \\
         &   -\left(\frac{2 \sqrt{\phi_{\max} (m)}}{1-k} +\sqrt{(s-2)\left(1-\frac{s-1}{T}\right)}\right) s \delta_0 - \sqrt{(s-2)\left(1-\frac{s-1}{T}\right)} \delta.
    \end{split}
\end{equation*}

Denote $\kappa = \sqrt{\phi_{min} (2)} - \sqrt{\phi_{\max} (m)} \frac{\sqrt{2}k}{1-k}$ and we complete the proof.
\end{proof}

% \begin{proof}[Proof of Lemma~\ref{lower_bound_lma}]
% Partition index set ${J_0}^{\mathsf{c}} = \cup_{\ell=1}^L J_\ell$, where $|J_1| = \dots = |J_{L-1}| = m$ and $|J_L|\leq m$, we will have
% \begin{equation*}
%     \begin{split}
%         \frac{1}{\sqrt{T}}\norm{\XX e}_2 &\geq \frac{1}{\sqrt{T}}\norm{\XX e_{J_0}}_2 - \frac{1}{\sqrt{T}}\norm{\XX e_{{J_0}^{\mathsf{c}}}}_2 \\
%         &\geq \sqrt{\phi_{min} (s)} \norm{e_{J_0}}_2 - \sqrt{\phi_{\max} (m)} \sum_{\ell=1}^L  \norm{e_{J_\ell}}_2\\
%         &\geq \sqrt{\phi_{min} (s)} \norm{e_{J_0}}_2 - \sqrt{\phi_{\max} (m)} \sum_{\ell=1}^L  \norm{e_{J_\ell}}_1\\
%         &= \sqrt{\phi_{min} (s)} \norm{e_{J_0}}_2 - \sqrt{\phi_{\max} (m)}  \norm{e_{{J_0}^{\mathsf{c}}}}_1.
%     \end{split}
% \end{equation*}

% Note that $I_3 = {J_0}^{\mathsf{c}}$ and by Lemma~\ref{restricted_error}, we have
% \begin{equation}
%     \begin{split}
%         \frac{1}{\sqrt{T}}\norm{\XX e}_2 &\geq 
%          \sqrt{\phi_{min} (s)} \norm{e_{J_0}}_2 - \sqrt{\phi_{\max} (m)}  \left(\frac{2s \delta_0}{1-k} + \frac{k}{1-k}\norm{e_{I_1}}_1\right)\\
%          &\geq \sqrt{\phi_{min} (s)} \norm{e_{J_0}}_2 - \sqrt{\phi_{\max} (m)}  \left(\frac{2s \delta_0}{1-k} + \frac{\sqrt{2}k}{1-k} \norm{e_{I_1}}_2\right)\\ 
%          &\geq \left(\sqrt{\phi_{min} (s)}  - \sqrt{\phi_{\max} (m)} \frac{\sqrt{2}k}{1-k}\right) \norm{e_{J_0}}_2 -  \frac{2s \delta_0 \sqrt{\phi_{\max} (m)}}{1-k}\\
%          & = \kappa \norm{e_{J_0}}_2 -  \frac{2s \delta_0 \sqrt{\phi_{\max} (m)}}{1-k}.
%     \end{split}
% \end{equation}
% We complete the proof.
% \end{proof}

\begin{proof}[Proof of Lemma~\ref{upper_bound_lma}]
Since $\hat \beta_T$ is solution to $\mathrm{VI}[ F_{\textbf{x}_{1:T}}, \mathcal{X}]$, the weak VI, and the vector field $F_{\textbf{x}_{1:T}}(\cdot)$ is continuous, we have $\hat \beta_T$ is also solution to the strong VI. That is, $\hat \beta_T$ also satisfies $$\langle  F_{\textbf{x}_{1:T}}(\hat \beta_T), w-\hat \beta_T\rangle \geq 0, \quad \forall w \in \mathcal{X}.$$
In particular, we have $\langle  F_{\textbf{x}_{1:T}}(\hat \beta_T), \beta-\hat \beta_T\rangle \geq 0$. Meanwhile, we have $F_{\textbf{x}_{1:T}}(\hat \beta_T) = F_{\textbf{x}_{1:T}}(\beta) - A[\textbf{x}_{1:T}] (\beta - \hat \beta_T)/T.$ Therefore, we will have $$\left \langle  F_{\textbf{x}_{1:T}}(\beta) - \frac{1}{T} A[\textbf{x}_{1:T}] (\beta - \hat \beta_T), \beta-\hat \beta_T \right\rangle \geq 0.$$

Rearrange terms and recall that $\eta = F_{\textbf{x}_{1:T}}(\beta)$, we will get
\begin{equation}
    (\beta - \hat \beta_T)^\T (A[\textbf{x}_{1:T}]/T) (\beta - \hat \beta_T) \leq \langle  \eta , \beta-\hat \beta_T\rangle \leq \norm{\eta}_{\infty} \norm{\beta-\hat \beta_T}_1,
\end{equation}where the last inequality comes from Hölder's inequality.

Notice that $ A[\textbf{x}_{1:T}] =\XX^\T \XX$, we can re-express the inequality above as $$\frac{1}{\sqrt{T}}\norm{\XX e}_2^2 \leq \norm{\eta}_{\infty} \norm{e}_1  = \norm{\eta}_{\infty} \left(\norm{e_{J_0}}_1 + \norm{e_{{J_0}^{\mathsf{c}}}}_1\right) \leq \frac{2}{1-k} \norm{\eta}_{\infty} \norm{e_{J_0}}_1,  $$where the last inequality comes from Lemma~\ref{restricted_error}.

By \eqref{Delta_bound} and the choice of $k\lambda$, we get$$\norm{\eta}_{\infty}\leq \frac{2\sqrt{2} A_3 \sigma_0^2(c_1+1)}{1-\alpha_1} \log(2T)\sqrt{\frac{\log(2T)}{T}}= O\left((\log T)^{3/2}/ T^{1/2}\right).$$

We complete the proof.
\end{proof}

\subsection{Proof of Theorem~\ref{thm_main_lower}}

\begin{proof}[Proof of Theorem~\ref{thm_main_lower}]
By Proposition~\ref{thm_lower_bound}, to make $\ell_2$ error lower bounded by $C_2$ with probability greater than $1-C_6$, we need \begin{equation}\label{lower_bound_order}
%\begin{array}{rl}
    C_2 = \frac{1}{2} \exp\left\{-\frac{C_3 T + C_4\delta_0(T)\sum_{t=2}^T s(t)+ C_5\delta_0^2(T)\sum_{t=2}^T s^2(t) + \log 2}{C_6 s(T)}\right\}.
    %\end{array}
\end{equation}

Since $s(t)\leq t$, we will have a decreasing (w.r.t $t$) lower bound at approximately exponential rate. Thus, without any condition, the naive bound will be tighter compared to the one we just derive if $\sqrt{s(t)}\delta_0(t)$ goes to infinity. To make sure the lower bound $C_2$ we derive in \eqref{lower_bound_order} is of constant order, we need $s(t)$ at least of order $t$, i.e. condition \eqref{sparsity}. However, this makes $\sum_{t=2}^T s^2(t) = \Theta(T^3)$ and we further need $\delta_0(t)$ small enough when $t \in \{1,\dots,T_0\}$, i.e. condition \eqref{accuracy}. 
\end{proof}

\begin{proof}[Proof of Proposition~\ref{thm_lower_bound}]

First, we find a large enough $\varepsilon$-packing by the following Lemma.

\begin{lemma}\label{packing_num}Let $(V, \norm{\cdot})$ be a normed space. For $\Theta \subset V \subset \RR^d$, we have
$$
\left(\frac{1}{\varepsilon}\right)^{d} \frac{\operatorname{vol}(\Theta)}{\operatorname{vol}(B)} \leq  N(\Theta,\|\cdot\|, \varepsilon) \leq \left(\frac{3}{\varepsilon}\right)^{d} \frac{\operatorname{vol}(\Theta)}{\operatorname{vol}(B)},
$$where $B$ is the unit norm ball and $$
N(\Theta,\|\cdot\|, \varepsilon)=\max \{m: \exists \  \varepsilon \text {-packing of } \Theta \text { of size } m\}
$$ is the packing number.
\end{lemma}

Recall that the coefficient vector space is $\Theta_T = \{\beta: (\alpha_1,\mu) \in \mathcal{S}, \Delta \in \mathcal{B} \}$. Since $\delta_s$ is constant, $\Theta_T$ will have a constant order volume, even though $\delta_0$ can be very small. Thus, by Lemma~\ref{packing_num} we can find an $\varepsilon$-packing $\cN = \{\beta_1,\dots,\beta_N\} \subset \Theta_T$ such that \begin{equation}\label{ep_packing}
    N \geq C_7  \left(\frac{1}{\varepsilon}\right)^s,
\end{equation}
% $$\varepsilon = \Theta (s^{2-\gamma}\delta_0), \quad N \geq C_3 \left\lfloor \frac{T-2}{s} \right\rfloor \left(\frac{1}{\varepsilon}\right)^s,$$
where $C_7$ is some positive constant.
% and $\lfloor \cdot \rfloor$ is the floor function.

\begin{lemma}\label{KL_upper_bound}
For any $\varepsilon$-packing $\cN = \{\beta_1,\dots,\beta_N\} \subset \Theta_T$, if the random noise is normally distributed, then the upper bound on KL divergence between the joint distributions of $ x_{1:T}$ generated by \eqref{DGP2} with coefficient chosen from $\cN$ is
$$\max_{i,j \in [N]} KL(\mathbf{p}_{\beta_i}||\mathbf{p}_{\beta_j}) \leq C_3 T + C_4\delta_0(T)\sum_{t=2}^T s(t)+ C_5\delta_0^2(T)\sum_{t=2}^T s^2(t),$$where $\mathbf{p}_{\beta}$ is joint probability density function (p.d.f.) of $ x_{1:T}$ generated by \eqref{DGP2} with coefficient $\beta$ and $C_3$, $C_4$ and $C_5$ are some positive constants dependent on $\delta_s$.
\end{lemma}

\begin{lemma}[Fano's inequality] Let $\cP = \{P_1,\dots,P_N\}$. For any random variable $Z$ taking values in $[N]$, we have \begin{equation}\label{fano}
    \frac{1}{N}\sum_{i=1}^N P_i\left(Z \not = i \right)\geq 1- \frac{\frac{1}{N^2}\sum_{i,j \in [N]} KL(P_i||P_j) +\log 2}{\log N},
\end{equation}where $KL(\cdot||\cdot)$ is the Kullback–Leibler (KL) divergence
\end{lemma}

By Fano's inequality \eqref{fano}, we have for any r.v. $Z$\begin{equation} \label{fano_ineq}
    \frac{1}{N}\sum_{i=1}^N \mathrm{pr}_{\beta_i}\left(Z \not = i \right)\geq 1- \frac{\max_{i,j \in [N]} KL(\mathbf{p}_i||\mathbf{p}_j) +\log 2}{\log N}.
\end{equation}

For any estimator $\Tilde{\beta}_T$, define \begin{equation}\label{def_psi}
    \hat \psi = \psi(\Tilde{\beta}_T) = \argmin_{i \in [N]} \norm{\Tilde{\beta}_T - \beta_i}_2,
\end{equation}which is the index for the element closest to $\Tilde{\beta}_T$ (in $\ell_2$ norm sense) in the $\varepsilon$-packing $\cN$.

Therefore, for any $\hat \psi \not = i$, we have\begin{equation*}
    \begin{split}
        \norm{\Tilde{\beta}_T - \beta_i}_2 \geq \norm{\beta_{\hat\psi} - \beta_i}_2 - \norm{\Tilde{\beta}_T - \beta_{\hat\psi}}_2\geq \norm{\beta_{\hat\psi} - \beta_i}_2 - \norm{\Tilde{\beta}_T - \beta_i}_2,
    \end{split}
\end{equation*} 
where the last inequality comes from \eqref{def_psi}. 

Re-arrange terms in the inequality above and we will have $$\norm{\Tilde{\beta}_T - \beta_i}_2 \geq \frac{1}{2}\norm{\beta_{\hat\psi} - \beta_i}_2 \geq \frac{\varepsilon}{2},$$where the last inequality comes from the definition of $\varepsilon$-packing, i.e. $$\min_{i\not=j} \norm{\beta_i - \beta_j}_2 >\varepsilon.$$

This means when $\beta = \beta_i$, event $\{\hat \psi \not = i\}$ is subset of event $\{\norm{\Tilde{\beta}_T - \beta_i}_2 \geq  \varepsilon/2\}$. Therefore, we have
\begin{equation*}
    \begin{split}
        \sup_{\beta \in \Theta_T} \mathrm{pr}_{\beta}\left(\norm{\Tilde{\beta}_T-\beta}_2 \geq \varepsilon/2 \right)
        &\geq \sup_{\beta \in \cN} \mathrm{pr}_{\beta}\left(\norm{\Tilde{\beta}_T-\beta}_2 \geq \varepsilon/2 \right) \\
        & \geq \max_{i \in [N]} \mathrm{pr}_{\beta_i}\left(\hat \psi \not = i \right)\geq \frac{1}{N}\sum_{i=1}^N \mathrm{pr}_{\beta_i}\left(\hat \psi \not = i \right).
    \end{split}
\end{equation*}
% we will have$$\{\hat \psi \not = i\} \subset\{\norm{\Tilde{\beta}_T - \beta_i}_2 \geq  \varepsilon/2\}.$$

% Since $\varepsilon = \Theta (s^{2-\gamma}\delta_0)$, there exists a positive constant $C_1$ such that $\varepsilon/2 \geq C_1 s^{2-\gamma}\delta_0$. we will have$$\{\hat \psi \not = i\} \subset\{\norm{\Tilde{\beta}_T - \beta_i}_2 \geq  \varepsilon/2\} \subset  \{\norm{\Tilde{\beta}_T - \beta_i}_2 \geq  C_1 s^{2-\gamma}\delta_0\}.$$

% Therefore, we have
% \begin{equation*}
%     \begin{split}
%         \sup_{\beta \in \Theta_T} \mathrm{pr}_{\beta}\left(\norm{\Tilde{\beta}_T-\beta}_2 \geq C_1 s^{2-\gamma}\delta_0 \right)
%         &\geq \sup_{\beta \in \cN} \mathrm{pr}_{\beta}\left(\norm{\Tilde{\beta}_T-\beta}_2 \geq C_1 s^{2-\gamma}\delta_0 \right)\\
%         &\geq \sup_{\beta \in \cN} \mathrm{pr}_{\beta}\left(\norm{\Tilde{\beta}_T-\beta}_2 \geq \varepsilon/2 \right)\\
%         & = \max_{i \in [N]} \mathrm{pr}_{\beta_i}\left(\norm{\Tilde{\beta}_T-\beta_i}_2 \geq \varepsilon/2 \right)\\
%         & \geq \max_{i \in [N]} \mathrm{pr}_{\beta_i}\left(\hat \psi \not = i \right)\geq \frac{1}{N}\sum_{i=1}^N \mathrm{pr}_{\beta_i}\left(\hat \psi \not = i \right).
%     \end{split}
% \end{equation*}

% = \max_{i \in [N]} \mathrm{pr}_{\beta_i}\left(\norm{\Tilde{\beta}_T-\beta_i}_2 \geq \varepsilon/2 \right)
Taking $Z = \hat \psi$ and $\varepsilon = 2C_2$ in \eqref{fano_ineq_final}, by \eqref{ep_packing} and \eqref{fano_ineq}, we complete the proof.
\end{proof}

\begin{proof}[Proof of Lemma~\ref{KL_upper_bound}]

For $ x_{1:T}$ generated by \eqref{DGP2} with $\beta = (\alpha,\mu,\Delta_2,\dots,\Delta_T)^\T$, we can derive that \begin{equation}\label{xt_expression}
    x_t = \frac{1-\alpha^{t}}{1-\alpha} \mu + \sum_{i=2}^t \frac{1-\alpha^{t+1-i}}{1-\alpha} \Delta_i + \sum_{i=1}^t \alpha^{t-i} \varepsilon_i, \quad t= 1, \dots, T,
\end{equation}where for $t=1$ the second term is zero. We further denote $$\tau^{(t)} = \frac{1-\alpha^{t}}{1-\alpha} \mu + \sum_{i=2}^t \frac{1-\alpha^{t+1-i}}{1-\alpha} \Delta_i, \text{ and } B^{(t)} = \sum_{i=1}^t \alpha^{t-i} \varepsilon_i.$$

Therefore, if the random noise in \eqref{DGP2} is Gaussian, then the joint distribution for $ x_{1:T}$ will be $N(\tau,\Sigma)$, where $\tau = (\tau^{(1)},\tau^{(2)},\dots,\tau^{(T)})^\T$, $\Sigma = P_\alpha P_\alpha^\T$ and \begin{equation*}
P_\alpha = \left(\begin{array}{cccccc}
\alpha^0&  & & & & \\
\alpha^1& \alpha^0 &   & & \\
\alpha^2& \alpha^1 & \alpha^0 & & \\
\vdots& \vdots & & & \ddots \\
\alpha^{T-1}& \alpha^{T-2} &  \dots& \dots & \dots& \alpha^0  \\
\end{array}\right).
\end{equation*}

By some simple algebra, we will obtain $\det(\Sigma) = \det(P_\alpha P_\alpha^\T) = \det(P_\alpha)^2 = 1$ and \begin{equation*}
\Sigma^{-1} = \left(\begin{array}{ccccccc}
\alpha^2+1& -\alpha & & & & & \\
-\alpha& \alpha^2+1 & -\alpha  & & & \\
 &-\alpha& \alpha^2+1 & -\alpha & & \\
 & & \ddots&\ddots & \ddots \\
  &  &  & -\alpha& \alpha^2+1 & -\alpha\\
  &  &  & & -\alpha & 1 \\
\end{array}\right).
\end{equation*}

Arbitrarily choose two distinct coefficients from $\cN$. Without loss of generality, we denote them to be $\beta_i$ ($=1, 2$). Given $x_1$, denote the joint p.d.f. of $ x_{1:T}$ generated by \eqref{DGP2} with coefficient $\beta$ by $\mathbf{p}( x_{1:T}|x_1;\beta)$. For simplicity, we denote $\mathbf{p}( x_{1:T}|x_1;\beta_i) = \mathbf{p}_i$ for $\beta_i \in \cN$, $i = 1, \dots, N$. 

By the derivation above, $\mathbf{p}_i$ is joint p.d.f. of $N(\tau_i,\Sigma_i)$. Then the KL divergence between these two $(T-1)-$dimensional multivariate Gaussian distributions is\begin{equation} \label{KL_gaussian}
    \begin{split}
        &KL(N(\tau_1,\Sigma_1)||N(\tau_2,\Sigma_2)) 
=  \int \log \frac{p_1(x)}{p_2(x)}  p_1(x) d x\\
=&\int\frac{1}{2} \left[ \log \frac{\left|\Sigma_{2}\right|}{\left|\Sigma_{1}\right|}-\left(x-\tau_{1}\right)^\T \Sigma_{1}^{-1}\left(x-\tau_{1}\right)+\left(x-\tau_{2}\right)^\T \Sigma_{2}^{-1}\left(x-\tau_{2}\right)\right]  p_1(x) d x \\
=&\frac{1}{2} \log \frac{\left|\Sigma_{2}\right|}{\left|\Sigma_{1}\right|}-\frac{1}{2} \operatorname{tr}\left\{E\left[\left(x-\tau_{1}\right)\left(x-\tau_{1}\right)^\T\right] \Sigma_{1}^{-1}\right\}+\frac{1}{2} E\left[\left(x-\tau_{2}\right)^\T \Sigma_{2}^{-1}\left(x-\tau_{2}\right)\right] \\
=&\frac{1}{2} \log \frac{\left|\Sigma_{2}\right|}{\left|\Sigma_{1}\right|}-\frac{1}{2} \operatorname{tr}\left\{I_{T}\right\}+\frac{1}{2}\left(\tau_{1}-\tau_{2}\right)^\T \Sigma_{2}^{-1}\left(\tau_{1}-\tau_{2}\right)+\frac{1}{2} \operatorname{tr}\left\{\Sigma_{2}^{-1} \Sigma_{1}\right\} \\
=&\frac{1}{2}\left[\log \frac{\left|\Sigma_{2}\right|}{\left|\Sigma_{1}\right|}-T+\operatorname{tr}\left\{\Sigma_{2}^{-1} \Sigma_{1}\right\}+\left(\tau_{2}-\tau_{1}\right)^\T \Sigma_{2}^{-1}\left(\tau_{2}-\tau_{1}\right)\right].
    \end{split}
\end{equation}

Since $\det(\Sigma_1) = \det(\Sigma_2) = 1$, we have
\begin{equation}\label{KL}
    KL(\mathbf{p}_1||\mathbf{p}_2) = \frac{1}{2}\left[\operatorname{tr}\left\{\Sigma_{2}^{-1} \Sigma_{1}\right\}-T+\left(\tau_{2}-\tau_{1}\right)^\T \Sigma_{2}^{-1}\left(\tau_{2}-\tau_{1}\right)\right].
\end{equation}

On one hand, by the explicit form of $\Sigma_1$ as well as $\Sigma_2$, we can derive that the explicit form of the diagnoal elements of $\Sigma_{2}^{-1} \Sigma_{1}$. For $i=2,\dots,T-1$, we have
\begin{equation*}
    \begin{split}
        \left(\Sigma_{2}^{-1} \Sigma_{1}\right)_{i,i} &= \sum_{k=1}^{T-1} \left(\Sigma_{2}^{-1} \right)_{k,i}\left( \Sigma_{1}\right)_{k,i}\\
        & = -\alpha_2 \left(\left( \Sigma_{1}\right)_{i-1,i} + \left( \Sigma_{1}\right)_{i+1,i}\right) + (\alpha_2^2 + 1) \left( \Sigma_{1}\right)_{i,i}\\
        & = -\alpha_1\alpha_2 \left( 2 \frac{1-\alpha_1^{2i-2}}{1-\alpha_1^2} + \alpha_1^{2i-2}\right) + (\alpha_2^2 + 1) \frac{1-\alpha_1^{2i}}{1-\alpha_1^2}\\
        &\leq   \frac{\alpha_2^2 + 2|\alpha_1\alpha_2| +1 }{1-\alpha_1^2} \leq   \frac{ 3\delta_s^2+1 }{1-\delta_s^2}  .
    \end{split}
\end{equation*}

Similarly, we can derive the expression for $\left(\Sigma_{2}^{-1} \Sigma_{1}\right)_{1,1}$ and $\left(\Sigma_{2}^{-1} \Sigma_{1}\right)_{T,T}$ and upper bound them by some constant. This means all diagonal elements are bounded uniformly by a constant. Thus, we have \begin{equation}\label{KL_trace}
    \operatorname{tr}\left\{\Sigma_{2}^{-1} \Sigma_{1}\right\}-T \leq (c_6-1)T,
\end{equation}where constant $c_6$ is the uniform upper bound the constant and we can  show $c_6 > 1$.

On the other hand, for $ i=1, 2$, we have  $$|\tau_i^{(t)} |=\left| \frac{1-\alpha_i^{t-1}}{1-\alpha_i} \mu_i + \sum_{j=2}^t \frac{1-\alpha_i^{t+1-j}}{1-\alpha_i} \Delta_{i,j}\right| \leq \frac{\delta_s + s(t) \delta_0}{1-\delta_s}.$$

Therefore, we have $$|\tau_1 - \tau_2| \leq \frac{2}{1-\delta_s} (\delta_s + s(1) \delta_0,\dots, \delta_s + s(T) \delta_0)^\T = \frac{2}{1-\delta_s}(\delta_s \mathbf 1 + \delta_0  s_{1:T}) ,$$ where $\mathbf 1 \in \RR^{T}$ is the vector of ones, $ s_{1:T} = (s(1),\dots,s(T))^\T$ and the inequality is pointwise. 

Denote\begin{equation*}
\Tilde{\Sigma}_{2}^{-1} = \left(\begin{array}{ccccccc}
\alpha_2^2+1& |\alpha_2|& & & & & \\
|\alpha_2|& \alpha_2^2+1 & |\alpha_2| & & & \\
 &|\alpha_2|& \alpha_2^2+1 & |\alpha_2|& & \\
 & & \ddots&\ddots & \ddots \\
  &  &  & |\alpha_2|& \alpha_2^2+1 & |\alpha_2|\\
  &  &  & & |\alpha_2|& 1 \\
\end{array}\right),
\end{equation*}we can get\begin{equation*}
    \begin{split}
        & \left|\left(\tau_{2}-\tau_{1}\right)^\T \Sigma_{2}^{-1}\left(\tau_{2}-\tau_{1}\right) \right|\\
         \leq& \left|\tau_{2}-\tau_{1}\right|^\T \Tilde{\Sigma}_{2}^{-1}\left|\tau_{2}-\tau_{1}\right|\\
       =   & \left(\frac{2}{1-\delta_s}\right)^2 \left(\delta_s^2 \mathbf 1^\T \Tilde{\Sigma}_{2}^{-1} \mathbf 1 + 2 \delta_s \delta_0 \mathbf 1^\T \Tilde{\Sigma}_{2}^{-1}  s_{1:T} + \delta_0^2  s_{1:T}^\T \Tilde{\Sigma}_{2}^{-1}  s_{1:T}\right).
    \end{split}
\end{equation*}

We can upper bound the last three terms above as follows (notice that $s(1)=0$) :
\begin{equation*}
    \begin{split}
            & \mathbf 1^\T \Tilde{\Sigma}_{2}^{-1} \mathbf 1 \leq \left(1 + |\alpha_2|\right)^2 T;\\
        & \mathbf 1^\T \Tilde{\Sigma}_{2}^{-1}  s_{1:T} \leq (1 + |\alpha_2|)^2 (s(2)+\cdots+s(T))  ;\\
                & s_{1:T}^\T \Tilde{\Sigma}_{2}^{-1}  s_{1:T} =\sum_{i=2}^{T-2}  \left( |\alpha_2| (s(i)+s(i+2)) + (\alpha_2^2+1) s(i+1)\right) s(i+1) \\
        & \quad \quad \quad \quad \quad \quad \quad  +  (\alpha_2^2+1) s(2)^2 + |\alpha_2|(s(2)s(3) + s(T-1)s(T)) + s(T)^2.
    \end{split}
\end{equation*}
By \eqref{KL}, \eqref{KL_trace} and last four inequalities, we prove Lemma~\ref{KL_upper_bound}.
\end{proof}

\section{Additional experiments}\label{appendix:add_exp}
\subsection{Numerical simulation}

We set $\alpha_1 = 0.1, \sigma_0^2 = 0.1$, $\delta_0 = 0.1$ and choose $s \in \{20, 200, 1000\}$.
For each $s$, we plot $\hat \alpha_1$ and $\delta$ selected by Ljung-Box test with respect to time $T$. The result is in Figure~\ref{exp5}.

\begin{figure}[H]
%\vspace{-0.1in}
\centering
\subfigure{\includegraphics[width=\linewidth]{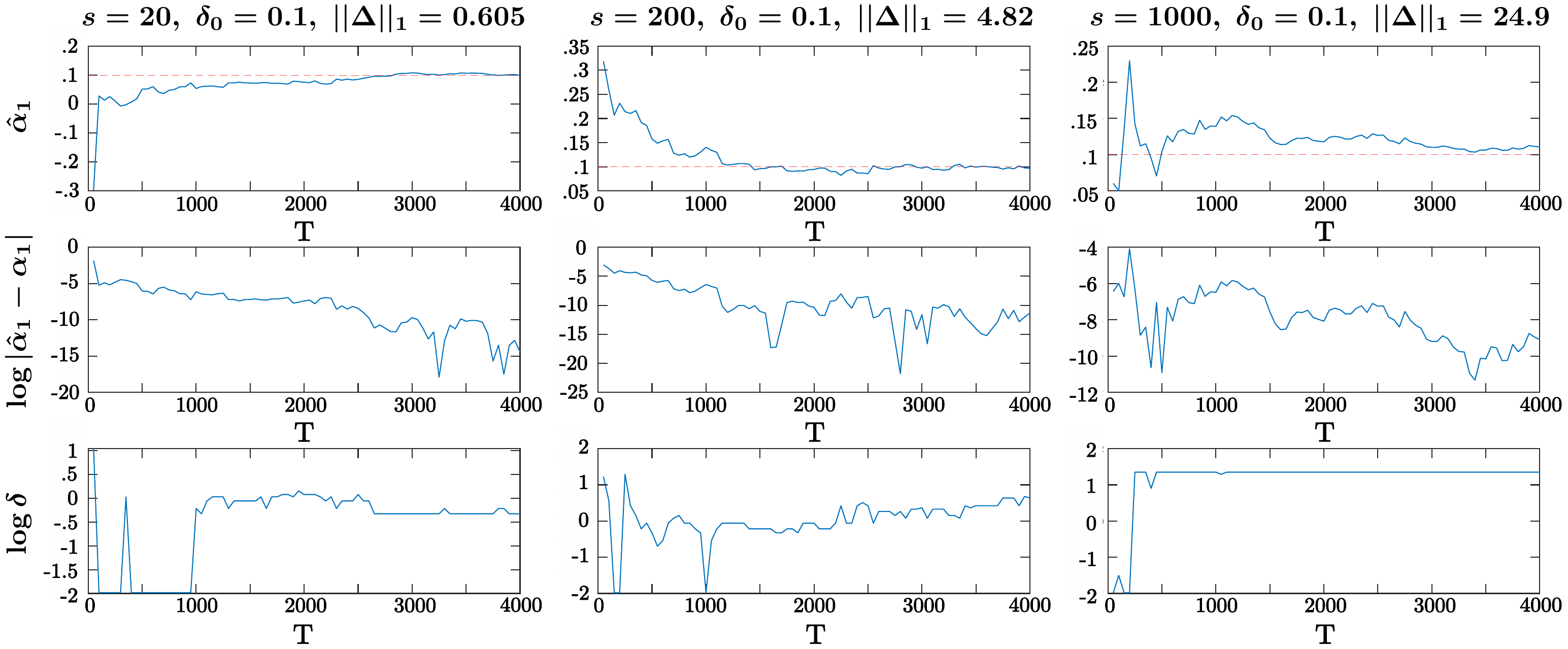}}
    %\vspace{-.2in}
    \caption{Algorithmic behavior with respect to $T$. The red dashed line is the ground truth $\alpha_1 = 0.1$. From the second row, we can observe that the estimation error converges to a larger value with increasing $s$.}\label{exp5}
    %\vspace{-0.1in}
\end{figure}

We have two main observations from Figure~\ref{exp5}: (i) the estimate $\hat \alpha_1$ will converge to an $\varepsilon$-optimal solution, but cannot converge to the ground truth and (ii) for larger $\norm{ \Delta}_1$, which is equivalent to larger $s$ and $\delta_0$, the estimation error after convergence will grow larger. Apart from this, we can see the behavior of the estimation error are similar to that of the tuning parameter $\delta$ selected by Ljung-Box test --- they converge at the same time. This validates our main theorem on the upper bound of the estimation error \eqref{upper_error_bound}. We also try more experimental settings ($s \in \{1000,2000,3000\}$ and $\delta_0 \in \{0.05,0.1\}$). We obtain similar results in Figure~\ref{exp6}.

\begin{figure}[htp]
%%\vspace{-0.1in}
\centering
\subfigure{\includegraphics[width=\linewidth]{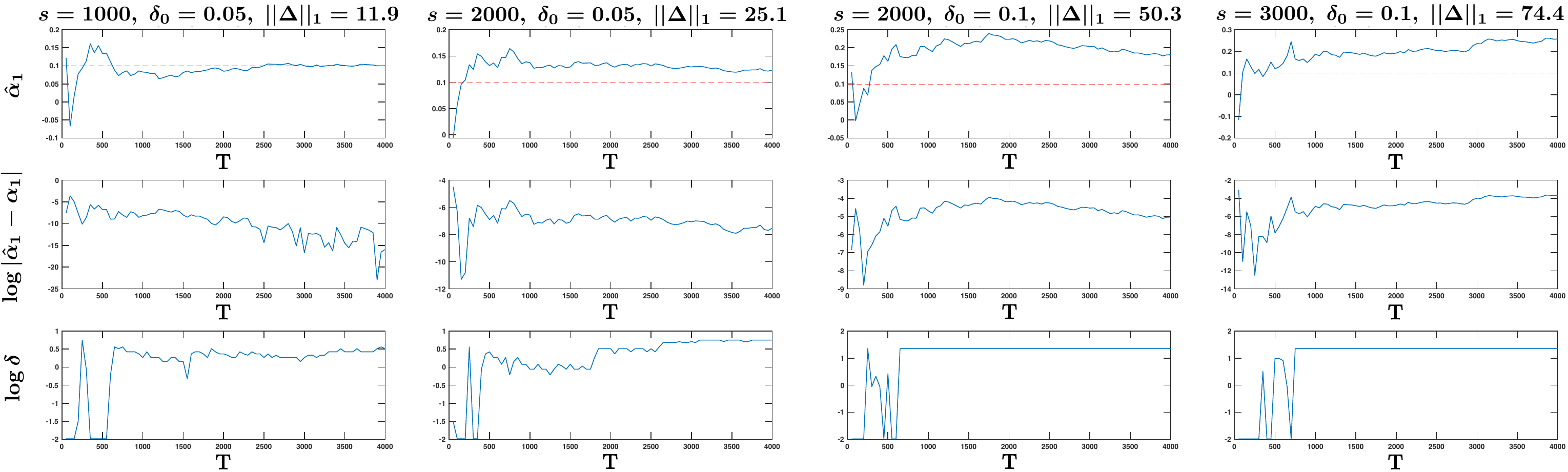}}
    %%\vspace{-.2in}
    \caption{In each column: the experimental settings, rate for sparse changes, one-step changes' magnitude and their total variation are listed on the top and the rest of the experimental settings are the same with Figure~\ref{exp5}; from top to the bottom, we plot $\hat \alpha_1$, logarithm of $\ell_2$ estimation error $\log |\hat \alpha_1 -\alpha_1|$ and logarithm of tuning parameter $\delta$ selected by Ljung-Box test with respect to $T$. In the fist column, the red dashed line is the ground truth $\alpha_1 = 0.1$. We can see that with increasing $s$ and  $\delta_0$, the estimation accuracy becomes lower.}\label{exp6}
    %%\vspace{-0.1in}
\end{figure}

\textit{Validation for a more general  \textsc{ar}$(p)$ case. } Here, we take  \textsc{ar}$(2)$ as an example.
We fix $\alpha_1 = 0.1, \sigma_0^2 = 0.1$ and $\delta_0 = 0.1$. We choose $s \in \{200, 1000, 2000, 3000\}$. Similarly, the dynamic background generating mechanism, estimation and parameter tuning procedure is the same as what we did in last section. We also apply Golden-section search (tolerance $\varepsilon$ is set to be 0.04) here. For each $s$, we plot the same algorithmic with respect to time $T$ in Figure~\ref{exp7}.

\begin{figure}[htp]
%%\vspace{-0.1in}
\centering
\subfigure{\includegraphics[width=\linewidth]{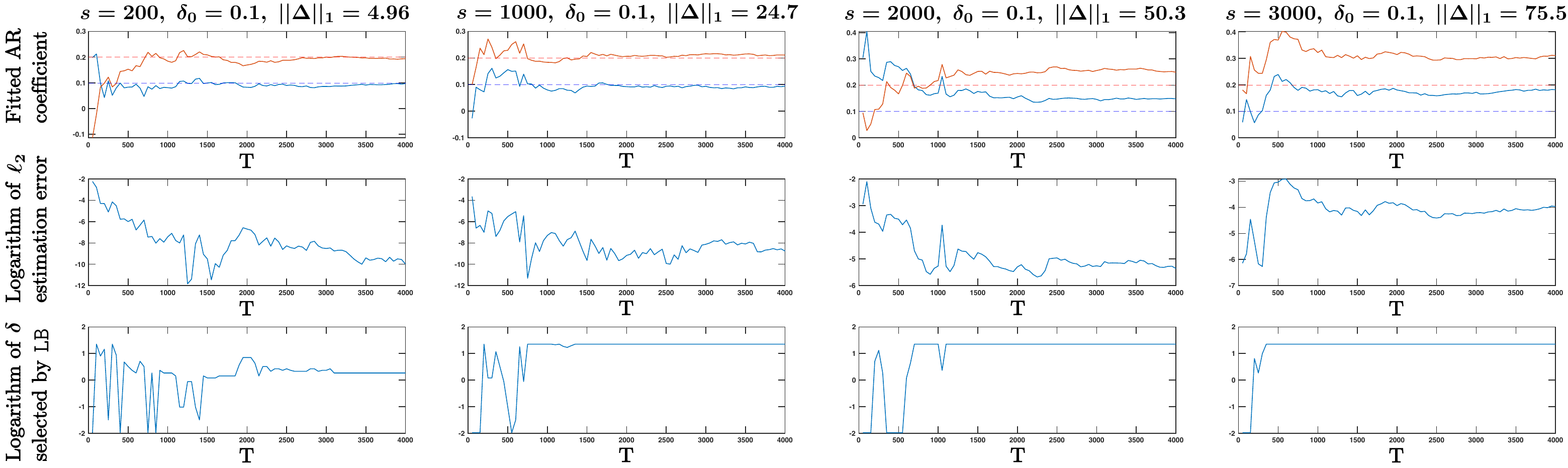}}
    %%\vspace{-.2in}
    \caption{In each column: the experimental settings, rate for sparse changes, one-step changes' magnitude and their total variation, are listed on the top; from top to the bottom, we plot $\hat \alpha_i \ (i=1, 2)$, logarithm of $\ell_2$ estimation error $\log \sqrt{(\hat \alpha_1 -\alpha_1)^2 + (\hat \alpha_2 -\alpha_2)^2}$ and logarithm of tuning parameter $\delta$ selected by Ljung-Box test with respect to $T$. In the fist column, the blue and red dashed line correspond to the ground truth $\alpha_1 = 0.1$ and $\alpha_1 = 0.2$, respectively. We can see that with increasing $s$ and  $\delta_0$, the estimation accuracy becomes lower, which is the same with  \textsc{ar}$(1)$ case.}\label{exp7}
    %%\vspace{-0.1in}
\end{figure}

We can see the results are similar to that of Figures~\ref{exp5} and~\ref{exp6}. Similarly to the analysis above, we validate our theoretical findings for  \textsc{ar}$(2)$ case. 

\textit{Comparison with polynomial variant.} Apart from piecewise constant function class, polynomial is another highly expressive function class. \citet{xu2008bootstrapping} proposed to use $n$th order polynomials (\texttt{n-poly}) to approximate the unstructured dynamics in non-stationary autoregressive time series. Then the autoregressive coefficients and polynomial coefficients are estimated via ordinary least square (OLS). However, he did not give instructions on how to choose $n$ in practice. Here, we choose $n \in \{3,5,10\}$ and compare \texttt{n-poly} with our proposed methods under the setting: $\alpha_1 = 0.1$, $\sigma_0^2 = 0.1$, $s=2000$, $\delta_0=0.05$, $\norm{ \Delta}_1 = 24.9$. The results are plotted in Figure~\ref{fig:exp_comparison_extended}.

\begin{figure}[htp]
%%\vspace{-0.1in}
\centering
\subfigure{\includegraphics[width=\linewidth]{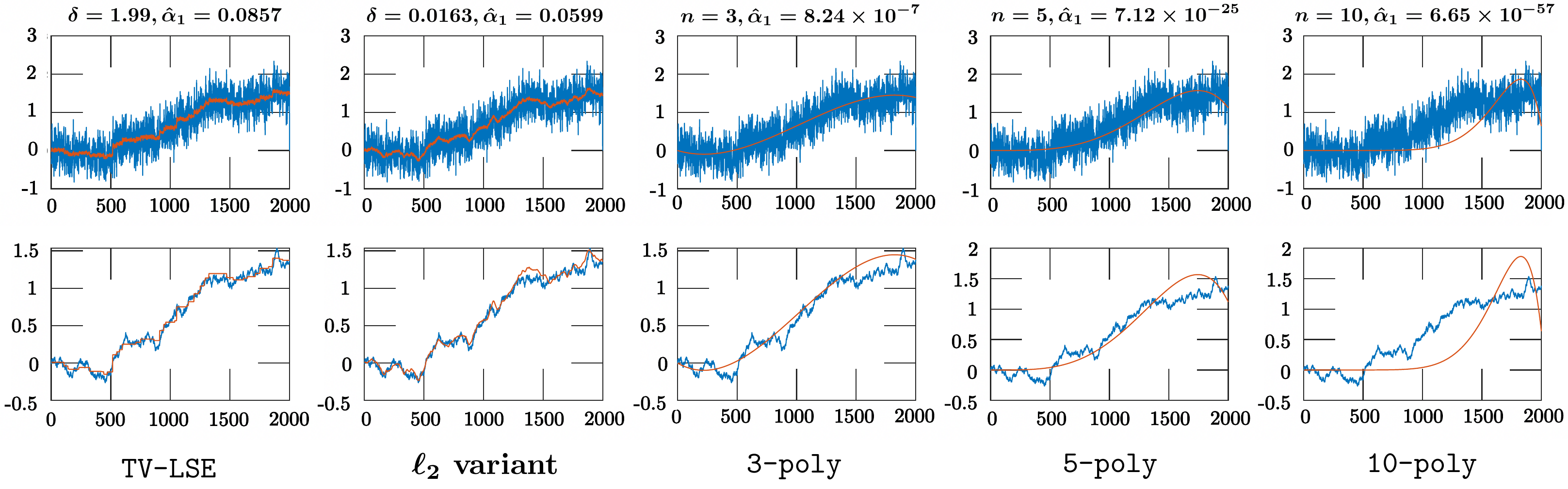}}
    %\vspace{-.05in}
    \caption{Comparison to \texttt{n-poly} due to \citet{xu2008bootstrapping} with $n \in \{3,5,10\}$ . The corresponding hyperparameters and  \textsc{ar}$(1)$ estimate are on the top of each column. We can see \texttt{n-poly} yields a very biased $\hat \alpha_1$ (even though \texttt{3-poly} faithfully captures the dynamics). }\label{fig:exp_comparison_extended}
    %%\vspace{-0.1in}
\end{figure}

From the figure above, we can see that all three polynomial methods considered here do not yield accurate estimate for \textsc{ar}$(1)$ series with highly unstructured dynamics. This is not surprising since polynomials are less expressive compared to piecewise constant function. Obviously, \texttt{n-poly} will perform better when the dynamics is smoother and more structured.

\subsection{Detailed estimation procedure in real data experiment}

Here, we take subject 23 as an example to show why we choose to use logarithm transform in detail. First, we directly apply our proposed estimator on the RT sequence with hyperparameter selected by Ljung-Box test, as is detailed in proposed tuning procedure. Since we do not have the ground truth, we can only access the goodness-of-fit by assessing how close our residual sequence resembles white noise. We plot the histogram as well as the QQ-plot of the fitted residual sequence. These two plots are shown in the first row in Figure~\ref{fig:realexp_rt_subj23}.

\begin{figure}[htp]
%%\vspace{-0.1in}
\centering
\subfigure{\includegraphics[width=\linewidth]{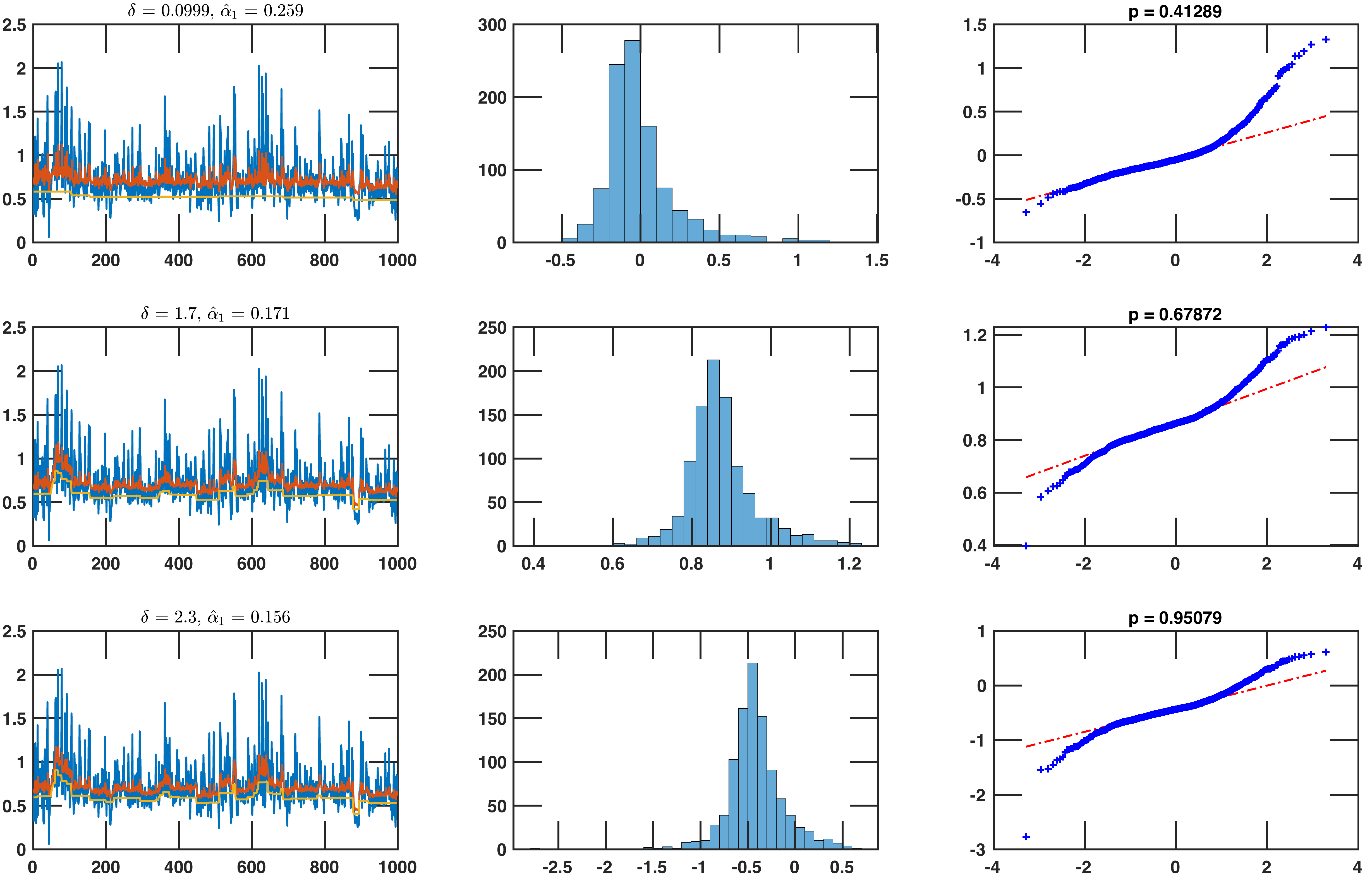}}
    %\vspace{-.05in}
\caption{Experimental results of applying our proposed estimator to subject 23 with hyperparameter $\delta_{ori}$ (top), $\delta_{cr}$ (middle) and $\delta_{log}$ (bottom). The first column plots raw observation (blue), fitted  \textsc{ar}$(1)$ model (red) and fitted dynamic background (yellow) with hyperparameter $\delta$ and estimated  \textsc{ar}$(1)$ coefficient $\hat \alpha_1$ on the top; the second and third column plot the histogram and quantile-quantile (QQ) plot of (original, cube root of and logarithm of) residuals (with $p_{ori},p_{cr},p_{log}$ on the top). }\label{fig:realexp_rt_subj23}
    %%\vspace{-0.1in}
\end{figure}

The histogram shows that the residuals are right-skewed --- in fact this is true for nearly all subjects. Ljung–Box test is commonly used in autoregressive integrated moving average (ARIMA) modeling, which requires Gaussian random noise assumption, and clearly this assumption breaks in this study. Therefore, the $p$-value of Ljung-Box test directly applied to residual sequence may not be a reasonable metric for the goodness-of-fit, which undermines the validity of $\delta$ selected by Ljung-Box test. Nevertheless, testing for remaining serial correlation in the residual sequence is the ultimate goal of applying Ljung-Box test. Thus, we can transform the residuals to more closely approximate a Gaussian distribution and then apply the Ljung-Box test on the transformed residuals to check for serial correlation. 

For right-skewed data, the most commonly used transforms are cube root and logarithm. We apply both transforms here. The transforms are performed by first subtracting $1.1 \times \min$ residuals from the residual sequence (to make sure we obtain meaningful values after logarithm), and then applying cube root or logarithm transform to this sequence.

We perform the aforementioned hyperparameter tuning procedure inn proposed tuning procedure for original and transformed residuals. More precisely, the $p$-value in step 2.(ii) is obtained by applying Ljung-Box test on original, cube root and logarithm of residuals. For each method, we denote the selected hyperparameter $\delta$ and the maximum of $p$-value to be $(\delta_{ori},p_{ori}), (\delta_{cr},p_{cr}), (\delta_{log},p_{log})$, respectively. We illustrate all these three methods on subject 23 by plotting the fitted  \textsc{ar}$(1)$ model, fitted dynamic background, histogram and QQ-plot of the residual sequence in Figure~\ref{fig:realexp_rt_subj23}.

% We perform the Ljung-Box test on the original, the cube root and logarithm of residual sequence in step 2.(ii) to obtain $p$-value (with the rest being the same as the procedure detailed in proposed tuning procedure); we obtain $(\delta_{ori},p_{ori}), (\delta_{cr},p_{cr}), (\delta_{log},p_{log})$ as illustrated in Figure~\ref{fig:realexp_rt_subj23_maxp}; we call the above three methods to tune hyperparameter $\delta$ as Ljung-Box test, $\texttt{LBcr}$ and $\texttt{LBlog}$, respectively. For each method, we plot the fitted  \textsc{ar}$(1)$ model, fitted dynamic background, histogram and QQ-plot of the residual sequence in Figure~\ref{fig:realexp_rt_subj23}.

Figure~\ref{fig:realexp_rt_subj23} shows that for subject 23 (i) from the first column, the first method clearly underfits the dynamic background; (ii) from the second column, the last histogram is much more symmetric and closely resembles p.d.f. of normal distribution; (iii) from the third column, the last method has larger $p$-value, indicating less serial correlation remained in residual sequence. This again shows that why we use $p$-value to select the hyperparameter --- it is a easy-to-use metric which correctly indicates whether the dynamic background is fitted properly. Moreover, we see that the third method, i.e. using logarithm transform, is the best for subject 23. In fact, logarithm transform the best for almost all subjects in the sense that $p_{log}$ is the largest among $p_{ori},p_{cr},p_{log}$. We also observe that for those subjects that $p_{log}$ is not the largest, the tuning parameter $\delta$ selected by all three methods are the same. Therefore, we adopt logarithm transform in our real data experiment.

\end{document}